\documentclass{amsart}
\usepackage{amsmath,amsthm,amssymb}
\usepackage{graphicx}
\usepackage{geometry, xcolor}
\usepackage{blkarray}
\usepackage{layout}
\usepackage{longtable}
\usepackage{comment}
\usepackage{breqn}
\usepackage{enumitem}
\usepackage{setspace}
\usepackage{centernot}
\usepackage[normalem]{ulem}
\usepackage{fullpage}
\usepackage{booktabs}
\usepackage{soul}
\usepackage{blkarray}
\usepackage{url}
\usepackage{tikz}
\usetikzlibrary{positioning}

\newcommand{\tocless}[2]{\bgroup\let\addcontentsline=\nocontentsline#1{#2}\egroup}

\usepackage{frcursive, pifont}

\newtheorem{theo}{Theorem}[section]

\newtheorem{rem}[theo]{Remark}
\newtheorem{cor}[theo]{Corollary}
\newtheorem{lem}[theo]{Lemma}
\newtheorem{prop}[theo]{Proposition}

\newtheorem{conj}[theo]{Conjecture}

\theoremstyle{definition}
\newtheorem{eg}[theo]{Example}

\newcommand{\id}{Id}
\newcommand{\BF}[1]{\mathbf{#1}}

\title{Estimating and computing Kronecker Coefficients: a vector partition function approach}
\author{Marni Mishna and Stefan Trandafir}

\begin{document}
\maketitle
\begin{abstract}
    We study the Kronecker coefficients $g_{\lambda, \mu, \nu}$ via a formula that was described by Mishna, Rosas, and Sundaram, in which the coefficients are expressed as a signed sum of vector partition function evaluations. In particular, we use this formula to determine formulas to evaluate, bound, and estimate $g_{\lambda, \mu, \nu}$ in terms of the lengths of the partitions $\lambda, \mu$, and $\nu$. We describe a computational tool to compute Kronecker coefficients $g_{\lambda, \mu, \nu}$ with $\ell(\mu) \leq 2,\ \ell(\nu) \leq 4,\ \ell(\lambda) \leq 8$. We present a set of new vanishing conditions for the Kronecker coefficients by relating to the vanishing of the related atomic Kronecker coefficients, themselves given by a single vector partition function evaluation. We give a stable face of the Kronecker polyhedron for any positive integers $m,n$. Finally, we give upper bounds on both the atomic Kronecker coefficients and Kronecker coefficients.
\end{abstract}
\tableofcontents

\section{Introduction}

\subsection{Computational questions on Kronecker coefficients}
The \emph{Kronecker coefficients} $g_{\lambda, \mu, \nu}$ are the structure constants in the decomposition of a tensor product of irreducible representations of the symmetric group into irreducible representations:
\begin{equation}
    V_{\mu} \otimes V_{\nu} = \bigoplus\limits_{\lambda} g_{\lambda, \mu, \nu} V_{\lambda}.
\end{equation}
Consequently, they can be expressed using Schur functions
\begin{equation} \label{eq:schur}
    s_{\lambda}[XY] = \sum_{\mu, \nu} g_{\lambda, \mu, \nu} s_{\mu}[X]s_{\nu}[Y],
\end{equation}
where $X := (x_1, \dots, x_m), Y := (y_1, \dots, y_n), XY := (x_1y_1, x_1y_2, \dots, x_my_n)$. 
Here, the Schur functions are indexed by partitions $\lambda, \mu, \nu$ with at most $mn,m,n$ parts respectively. The number of parts of a partition $\alpha$ is its \emph{length} and is denoted here by $\ell(\alpha)$. When we focus on the set of partitions of length at most $k$, we may pad the partitions with $0$'s so that they are all sequences of length $k$. 

Since their introduction in 1938 by Murnaghan, the Kronecker coefficients have proved to be among the most intriguing objects in algebraic combinatorics. After several decades of research, many open questions about the Kronecker coefficients remain. They are all non-negative integers, but have no known combinatorial interpretation. The Littlewood-Richardson $c_{\mu, \nu}^{\lambda}$ coefficients describe the decomposition of a tensor product of irreducible representations of the \emph{general linear group} into irreducible representations. This difference could appear small, yet the Littlewood-Richardson $c_{\mu, \nu}^{\lambda}$ coefficients have several combinatorial interpretations. One might view Kronecker coefficients as a generalization of the Littlewood-Richardson coefficients, hence the resistance to a clear interpretation is surprising, particularly in view of the publicity they have received\footnote{``One of the main problems in the combinatorial representation theory of the symmetric group is to obtain a combinatorial interpretation for the Kronecker coefficients." - Stanley \cite{Stan99}}. 

The basic problem of computing the Kronecker coefficient $g_{\lambda, \mu, \nu}$ for general partitions $\lambda, \mu, \nu$ is $\#$P-hard (in the bitlength of the size of the partitions) and contained in GapP~\cite{BuIk08}\footnote{Problems in GapP can be expressed as the difference of two functions which are in $\#$P.}.  Baldoni, Vergne and Walter distribute code~\cite{BaVeWa17} for use with  Maple mathematical software to compute Kronecker coefficients for partitions $\lambda, \mu, \nu$ of bounded lengths ($\ell(\lambda), \ell(\mu), \ell(\nu) \leq 3$; and  $\ell(\lambda) \leq 6, \ell(\mu) \leq 2, \ell(\nu) \leq 3$). There are at least two packages that handle partition sets without a bound on length, such as~\cite{ChDoWa12} and the SF Maple package of Stembridge~\cite{Stem05}. These two packages are prohibitably computationally expensive except for some small, or particular cases.  

Even the problem of approximating the Kronecker coefficients is non-trivial and very few useful bounds are known.  Pak and Panova~\cite[Corollary 3.4]{PaPa14} determine a bound for partitions $\lambda, \mu, \nu$ with $\ell(\lambda) \leq l, \ell(\mu) \leq m, \ell(\nu) \leq n$:
\begin{equation} \label{eq:pp-bound1}
    g_{\lambda, \mu, \nu} \leq \prod_{i=1}^{l} \binom{\lambda_i - i + mn}{mn - i}.
\end{equation}
More recently, in~\cite{PaPa20}, they obtained the following bound in $N = |\lambda| = |\mu| = |\nu|$ via contingency tables
\begin{equation} \label{eq:pp-bound2}
    g_{\lambda, \mu, \nu} \leq \left(1 + \frac{lmn}{N}\right)^N\left(1 + \frac{N}{lmn}\right)^{lmn}.
\end{equation}
Remark both of these bounds are polynomial in the length of the partitions - however, the degree is generally not optimal. For example, when $l = 4, m=2, n=2$, we find that if  $\lambda = (\frac{N}{4}, \frac{N}{4}, \frac{N}{4}, \frac{N}{4})$, both bounds are of the order $O(N^{16})$ whereas the actual growth is (more precisely) $O(N^2)$~\cite{BrOrRo09}. 

    By $h(u)$, we denote the hook-length of the box~$u$ in the Ferrers diagram of $\lambda$. The hook length formula~\cite[Corollary 3.2]{PaPa20} also gives a bound:
\begin{equation}
    g_{\lambda, \mu, \nu} \leq \min(f^{\lambda}, f^{\mu}, f^{\nu}),
\end{equation}
where $f^{\alpha} := \frac{k!}{\prod\limits_{u \in [\lambda]} h(u)}$ for a partition $\alpha$ of length $k$.

Some progress has been made to understand conditions on $\lambda, \mu, \nu$ for which $g_{\lambda, \mu, \nu}=0$. Denote by $k\alpha$ the partition obtained by multiplying each part of $\alpha$ by $k$. Littlewood-Richardson coefficients satisfy a \emph{saturation property}: 
\[c_{\lambda, \mu}^{\nu} = 0 \iff c_{k\lambda, k\mu}^{k\nu} = 0.\]The Kronecker coefficients do not satisfy such a property universally: 
\[g_{(1, 1), (1, 1), (1, 1)} = 0,\quad\text{ but }\quad g_{(2, 2),(2, 2),(2, 2)} = 1.\]

Even deciding ``$g_{\lambda, \mu, \nu} = 0$?" is NP-hard~\cite{IkMuWa17}. There are numerous vanishing conditions known- typically expressed as inequalities in the parts of $\lambda, \mu, \nu$ which guarantee that the coefficient $g_{\lambda, \mu, \nu}$ is zero. A classical result of Murnaghan and Littlewood (appearing for example in \cite{JaKe81}) is that for any non-zero Kronecker coefficient $g_{\lambda, \mu, \nu}$, it follows that $\overline{\lambda} \leq \overline{\mu} + \overline{\nu}$, where $\overline{\gamma}$ is the partition obtained by deleting the first part of partition $\gamma$. Consider the set of points constructed by concatenating partitions of fixed length. Those points that come from partitions giving a non-zero Kronecker coefficient have a nice geometry. Specifically, 
\begin{equation}
    Kron_{l,m,n} := \{(\lambda, \mu, \nu)\in \mathbb{Z}^{l+m+n} : g_{\lambda, \mu, \nu} \neq 0,\ \ell(\lambda) \leq l,\ \ell(\mu) \leq m, \ell(\nu)  \leq n\}
\end{equation} 
is a finitely generated semigroup in $\mathbb{Z}^{l+m+n}$ that generates a rational polyhedral cone. Following Manivel~\cite{Man15}, we call this cone the \emph{Kronecker polyhedron} and denote it $PKron_{l,m,n}$. In~\cite{Kl04} the cone $PKron_{l,m,n}$ is computed explicitly for small values of $l,m,n$, and it seems the number of inequalities increases rapidly. While this set is theoretically computable for any positive integers $l,m,n$, it is quickly computationally infeasible to do so. Another set of vanishing conditions valid for triples of partitions of any lengths were given recently by Ressayre in~\cite[Theorems 1 \& 2]{Re19}.

A classic result of Murnaghan states that for partitions $(\lambda, \mu, \nu)$ the sequence $\left(g_{\lambda + (k), \mu + (k) \nu + (k)}\right)_{k \geq 0}$ eventually stabilizes. Since then, many other partition triples $\alpha, \beta, \gamma$ with this property have been identified-- that is, the values of the sequence $\left(g_{\lambda + k\gamma, \mu + k\alpha, \nu + k\beta}\right)_{k \geq 0}$ stabilize for any choice of $\lambda, \mu, \nu$. Such triples $(\alpha, \beta, \gamma)$ are called \emph{stable triples}.
Stabilization phenomenon have been studied in \cite{BrOrRo11, Man15, Man15-2, Pell21, Stem14}.


Applications of Kronecker coefficients extend beyond the realm of algebraic combinatorics. The \emph{Geometric Complexity Theory} (GCT) program, developed by Mulmuley and Sohoni, with the goal of solving $P$ versus $NP$, relies heavily on the computation of Kronecker coefficients as one of its main ingredients (see \cite{BlMuSo15, IkMuWa17, Mu11}). More specifically, problems of positivity (as discussed in the Appendix of \cite{BrOrRo09-2} entitled \emph{Erratum to the saturation hypothesis (SH) in 
``Geometric Complexity Theory VI''} and contributed by Mulmuley) related to the previously described saturation problems play an important role.

Kronecker coefficients appear in quantum computing where they encode the relationship between composite systems and their subsystems \cite{ChDoWa12, ChHaMi07, ChMi06}. As in the case of GCT, being able to determine the positivity of Kronecker coefficients is useful. In the context of quantum computing, non-zero Kronecker coefficients correspond to \emph{admissible spectral triples} which play an important role in the study of bipartite quantum states in quantum information theory \cite{Ch06}.

\subsection{Kronecker coefficients and vector partition functions} \label{sec:KCandVP}

Here, we address many of these fundamental questions on Kronecker coefficients using a detailed analysis of Eq.~\eqref{eq:schur}. The first step is  to deduce an expression for $g_{\lambda, \mu, \nu}$ using coefficient extraction of multivariate Taylor series of rational functions. This formulation allows us to represent Kronecker coefficients as a signed sum of \emph{vector partition function} evaluations. Let~$A$ be a~$d \times n$ matrix~$A$ with integer entries satisfying \[\operatorname{ker}(A) \cap \mathbb{R}^n_{\geq 0} = \{\BF{0}\}.\] Denote the columns of $A$ by $\BF{a}_1, \dots, \BF{a}_n$. The \emph{vector partition function} $p_A : \mathbb{Z}_{\geq 0}^m \to \mathbb{Z}_{\geq 0}$ is the counting function 
\[
p_A(\BF{b}) :=  \# \{\BF{x} \in \mathbb{N}^d : A\BF{x} = \BF{b} \}.
\]

Geometrically, $p_A(\BF{b})$ is the discrete volume of the polyhedron defined by the solutions of $A\BF{x}=\BF{b}$ and the inequalities $x_i \geq 0$ for $i = 1,\dots,d$. The generating function formulation is in terms of the coefficient of the term $\BF{x}^{\BF{b}}$ in the Taylor series expansion of a product of geometric series:
\begin{equation}
    p_A(\BF{b}) = [\BF{x}^{\BF{b}}] \prod\limits_{j=1}^{n} \frac{1}{{1-\BF{x}}^{\BF{a}_j}},
\end{equation}
with the convention that for vectors $\BF{u},\BF{v} \in \mathbb{Z}^d$, $\BF{u}^{\BF{v}}$ denotes the product $\prod\limits_{i=1}^{d} u_i^{v_i}$. The rational function $\prod\limits_{j=1}^{n} \frac{1}{{1-\BF{x}}^{\BF{a}_j}}$ is called the \emph{vector partition generating function}.
Sturmfels~\cite{Stur94} determined that the vector partition function is a piecewise quasi-polynomial whose domains of quasi-polynomiality are convex polyhedral cones called \emph{chambers} of a fan called the \emph{chamber complex} of $A$ (defined by Zelevinsky, Alekseevskaya, and Gelfand \cite{AlGeZe87}). 
Barvinok gave an algorithm which allows one to compute $p_A(\BF{b})$ in polynomial-time for fixed dimension $n$ of the the polytope $A\BF{x} = \BF{b}, \BF{x} \geq \BF{0}$ \cite{Barv94} .  An adapted version of this algorithm (the Barvinok-Woods algorithm) has been implemented in C (the implementation is named \emph{Barvinok}) and is publicly available \cite{KoVeWo08}.

 
We reformulate the expression for~$g_{\lambda, \mu, \nu}$ given in~\cite[Theorem 26]{MiRoSu21} as Theorem~\ref{theo:vpf-to-kron} below. The main ingredients in this approach are:
\begin{enumerate}
    \item a matrix $A^{m,n}$ and its vector partition function $p_{A^{m,n}}$;
    \item vectors $\alpha, \beta$;
    \item linear functions $r_s, r_t$;
    \item linear functions $l_s(\cdot; \sigma), l_t(\cdot; \sigma)$ defined for each $\sigma \in \mathfrak{S}_{mn}$.
\end{enumerate}
The quantities $\alpha, \beta$, and the linear functions $r_s, r_t, l_s, l_t$ (which all depend on $m,n$) are explicitly given in the discussion after Theorem \ref{theo:vpf-to-kron} and defined (implicitly) in \cite{MiRoSu21}. The matrices $A^{m,n}$ are given implicitly in~\cite{MiRoSu21}; we give explicitly only the cases $m=2$, $n =3,4$ used in our work. For given $m,n$ and $\sigma$, we call the function in the parts of $\lambda, \mu, \nu$, $\mathbf{b}^{m,n}(\lambda, \mu, \nu; \sigma) := (r_s(\mu, \nu) + \alpha - l_s(\lambda; \sigma),\ r_t(\mu, \nu) + \beta - l_t(\lambda; \sigma))$ the \emph{vector partition function input of $\sigma$}. Additionally, we refer to the quantity 
\begin{equation}
\operatorname{sgn}(\sigma)\ p_{A^{m,n}}\biggr(\mathbf{b}^{m,n}(\lambda, \mu, \nu; \sigma)\biggr)
\end{equation}
as the \emph{contribution of the alternant term associated to $\sigma$}.
In general it will be clear to which $m,n$ we refer, but we explicitly state this when needed.

\begin{theo} \label{theo:vpf-to-kron}
Let~$m,n$ be positive integers. Then for any partitions $\lambda, \mu, \nu$ with $\ell(\lambda) \leq mn,\ \ell(\mu) \leq m,\ \ell(\nu) \leq n$, we have
\begin{equation} \label{eq:vpf-to-kron}
    g_{\lambda, \mu, \nu} = \sum_{\sigma \in \mathfrak{S}_{mn}} \operatorname{sgn}(\sigma)\ p_{A^{m,n}}\biggr(\mathbf{b}^{m,n}(\lambda, \mu, \nu; \sigma)\biggr).
\end{equation}
\end{theo} 

The following expressions are valid for all integers $u,v$ with $1 \leq u \leq m-1$ and $1 \leq v \leq n-2$. The components of the vectors $\alpha, \beta$ are:
\begin{align*}
  \alpha_0 &= \frac{1}{2} \, {\left(n m + n - m - 2\right)} {\left(n - 1\right)} {\left(m - 1\right)} \\
  \alpha_u &=  \frac{1}{2} \, {\left(u^{2} n - 2 \, u n m + 2 \, n m^{2} - u^{2} + u - n - 2 \, m + 2\right)} {\left(n - 1\right)}\\
  \beta_v &= \frac{1}{12} \, {\left(8 \, n^{2} m^{2} - 6 \, v n m + 5 \, n^{2} m - 10 \, n m^{2} + 6 \, v^{2} - 12 \, v n + 6 \, v m - 19 \, n m + 2 \, m^{2} + 18 \, v + 14 \, m\right)} {\left(m - 1\right)}.
\end{align*}
The components of the vectors $r_s(\mu, \nu)$ and $r_t(\mu, \nu)$ are:
\begin{align*}
    r_s(\mu, \nu)_0 &= |\nu| - \nu_1 + \binom{n-1}{2} \\
    r_s(\mu, \nu)_u &= \sum_{i=u+1}^{m} \mu_{i} + |\nu| - \nu_1 + \binom{m-u}{2} + \binom{n-1}{2} \\
    r_t(\mu, \nu)_v &= \sum_{i=2}^{m} (i-1)\mu_{i} + (m-1)\sum_{j=2}^{v+1} \nu_{j} + m\sum_{j=v+2}^{n} \nu_{j} + \binom{m}{3} + (m-1)\binom{n-1}{2} + \binom{n-v-1}{2}.
\end{align*}
The components of the vectors $l_s(\lambda; \sigma)$ and $l_t(\lambda; \sigma)$ are: 
\begin{align*}
    l_s(\lambda; \sigma)_0 &= \sum_{i=m+1}^{mn} \biggr(\lambda_{\sigma(i)} + \delta_{\sigma(i)}\biggr) \\
    l_s(\lambda; \sigma)_u &= \sum_{i=u+1}^{m+u(n-1)} \biggr(\lambda_{\sigma(i)} + \delta_{\sigma(i)}\biggr) + 2\sum_{i=m+u(n-1) + 1}^{mn} \biggr( \lambda_{\sigma(i)} + \delta_{\sigma(i)}\biggr) \\
    l_t(\lambda; \sigma)_v &= \sum_{i=2}^{m} (i-1)\biggr(\lambda_{\sigma(i)} + \delta_{\sigma(i)}\biggr) + (m-1)\sum_{i=m+1}^{m+v} \biggr( \lambda_{\sigma(i)} + \delta_{\sigma(i)} \biggr) + m\sum_{i=m+v+1}^{m+n-1} \biggr( \lambda_{\sigma(i)} + \delta_{\sigma(i)}\biggr)\\
    &+ \sum_{i=1}^{m-1}\sum_{j=1}^{v} (i+m-1) \biggr( \lambda_{\sigma(m+i(n-1)+j)} + \delta_{\sigma(m+i(n-1)+j)} \biggr) + \sum_{i=1}^{m-1}\sum_{j=v+1}^{n-1} (i+m)\lambda_{m+i(n-1) + j}.
\end{align*}
The identity permutation in $\mathfrak{S}_{mn}$ is denoted by $\id$. It is convenient to have an explicit derivation in the case when $\sigma = \id$: 
\begin{align*}
    l_s(\lambda; \id)_0 &= \sum_{i=m+1}^{mn} \lambda_i + \frac{1}{2} \, {\left(n m - m - 1\right)} {\left(n - 1\right)} m \\
    l_s(\lambda; \id)_u &= \sum_{i=u+1}^{m+u(n-1)}\lambda_i + 2\sum_{i=m+u(n-1) + 1}^{mn} \lambda_i + \frac{1}{2}(m-u)\biggr(2mn - u - m - 1\biggr) \\
    &+ (n-1)^2m  - \binom{n}{2} + \frac{1}{2}(n-1)(m-1)\biggr(n(m-1) - m\biggr)\\
    &+ \frac{1}{2}(n-1)(m-u)\biggr(mn - u(n-1) - m - 1\biggr) \\
    l_t(\lambda; \id)_v &= \sum_{i=2}^{m} (i-1)\lambda_{i} + (m-1)\sum_{i=m+1}^{m+v}  \lambda_{i} + m\sum_{i=m+v+1}^{m+n-1} \lambda_{i}\\
    &+ \sum_{i=1}^{m-1}\sum_{j=1}^{v} (i+m-1)\lambda_{m+i(n-1)+j} + \sum_{i=1}^{m-1}\sum_{j=v+1}^{n-1} (i+m)\lambda_{m+i(n-1) + j} \\
    &+ \frac{m}{12} \left( 8\,{m}^{2}-3\,m+1 \right) {n}^{2} -m\left( m+1 \right) 
 \left( 10\,m+6\,v-1 \right) n \\
 &+2\,m\left( 3\,{v}^{2}+3\,vm+2\,{m}^{2}+6\,v+3\,m+1 \right).
\end{align*}

Notably, for all $0 \leq u\leq m-1$, the constant term (with respect to $\lambda_1, \dots, \lambda_{mn},\ \mu_1, \dots, \mu_m,\ \nu_1, \dots, \nu_n$) of the $i$\textsuperscript{th} coordinate of $r_s(\mu, \nu) + \alpha - l_s(\lambda; \id)$ is $0$, and for all $1 \leq v \leq n-2$, the constant term of the $j$\textsuperscript{th} coordinate of $r_t(\mu, \nu) + \beta - l_t(\lambda; \id)$ is also $0$. In other words both $r_s(\mu, \nu) + \alpha - l_s(\lambda; \id)$ and $r_t(\mu, \nu) + \beta - l_t(\lambda; \id)$ are linear forms whose variables are the parts of $\lambda, \mu, \nu$ (and thus so is the vector partition function input $\BF{b}^{m,n}(\lambda, \mu, \nu; \id)$).

The expression in Eq.~\eqref{eq:vpf-to-kron} writes the Kronecker coefficient $g_{\lambda, \mu, \nu}$ as a signed sum of permutations. The single term associated with the identity permutation is the largest, and can be used to derive properties about the Kronecker coefficient. Specifically, for partitions $\lambda, \mu, \nu$ with $\ell(\mu) \leq m,\ \ell(\nu) \leq n\ \ell(\lambda) \leq mn$, the \emph{atomic Kronecker coefficient} $\tilde{g}^{m,n}_{\lambda, \mu, \nu}$ is the coefficient obtained by taking only the contribution of the alternant term corresponding to the identity permutation in Eq.~\eqref{eq:vpf-to-kron} - that is,
\begin{equation}\label{eq:atomic}
    \tilde{g}^{m,n}_{\lambda, \mu, \nu} := p_{A^{m,n}}\biggr(\mathbf{b}^{m,n}(\lambda, \mu, \nu; \id)\biggr).
\end{equation}
Atomic Kronecker coefficients were introduced in~\cite{MiRoSu21}, where it was proven that in the $m=n=2$ case they provide an upper bound for the Kronecker coefficients. These authors also conjecture that they provide an upper bound in general~\cite{MiRoSu18}. This seems to be justified in each computation we have made (in the $m=2, n=3,4$ cases). Interestingly, the atomic Kronecker coefficients depend on the values $m,n$ and not just the indexing partitions. As an example (given also in \cite{MiRoSu21}), consider $\lambda = (12, 7, 4, 1), \mu = (12, 12), \nu = (12, 12)$. If we set, $m=n=2$, the atomic Kronecker coefficient $\tilde{g}^{2,2}_{\lambda, \mu, \nu}$ is $32$ - however, by by padding $\lambda$ and $\nu$ with zeroes (i.e. representing $\lambda, \nu$ as $\lambda = (12, 7, 4, 1, 0, 0), \nu = (12, 12, 0)$), we find that the atomic Kronecker coefficient $\tilde{g}^{2,3}_{\lambda, \mu, \nu}$ in this case is $8793$. The atomic Kronecker coefficients are expressed using a single partition function, which is polynomial time computable for a fixed dimension. However, the dimension grows very quickly as a function of $m,n$.

\subsection{Summary of contribution}

In this work we apply Theorem \ref{theo:vpf-to-kron} to study some of the main questions of Kronecker coefficients: exact computation, vanishing conditions, stability, and upper bounds. In \cite{MiRoSu21} the authors focused on the $m=n=2$ case; we adapt the main ideas of that article to general $m,n$. Section \ref{sec:VPF} describes the pertinent aspects of their work to this article. 

Once an expression of the vector partition function $p_{A^{m,n}}$ as a piecewise quasi-polynomial has been computed, the complexity of using this form to determine the Kronecker coefficient comes from the large number $(mn)!$ of terms in the sum. Significantly fewer than the $(mn)!$ terms are needed (either due to vanishing or cancellation): when $m=n=2$ only $7$ of the $24$ terms are needed, and when $m=2, n=3$ at most $482$ are needed. However, we do not know how many terms are needed in general for a given $m,n$. 

Using this to compute $g_{\lambda, \mu, \nu}$ is efficient for small $m$ and $n$, and we have developed a \emph{Sagemath} tool for computing Kronecker coefficients $g_{\lambda, \mu, \nu}$ for $l \leq 8, m \leq 2, n \leq 4$. The exact formulas are given in Section \ref{sec:explicit-computation}. This section is split into two subsections - in \ref{sec:2-3-6-case} we describe the more restricted case $\ell(\mu) \leq 2, \ell(\nu) \leq 3, \ell(\lambda) \leq 6$, and in \ref{sec:2-4-8-case} we describe the general case.

In Section \ref{sec:gen-bravyi-vanish}, we show vanishing conditions (conditions on $\lambda, \mu, \nu$ ensuring that the coefficient in question is $0$) on the atomic Kronecker coefficient give vanishing conditions for the Kronecker coefficients. We subsequently deduce explicit conditions. These are given in Theorem \ref{theo:gen-bravyi}. For each $m,n$ we obtain a set of $m + n - 2$ conditions for partitions $\lambda, \mu, \nu$ with $\ell(\mu) \leq m,\ \ell(\nu) \leq n,\ \ell(\lambda) \leq mn$. Our conditions have the advantage of being easy to compute and implement practically.

By considering the set of partition triples $(\lambda, \mu, \nu)$ satisfying the equation
\begin{equation} \label{eq:atomic-origin}
    \BF{b}^{m,n}(\lambda, \mu, \nu; \id) = \BF{0}
\end{equation}
we obtain a stable face of the Kronecker cone $PKron_{m,n,mn}$. Additionally, each $(\lambda, \mu, \nu)$ satisfying the above equation has the property that $g_{\lambda, \mu, \nu} = 1$, and moreover the partition triple is stable (Theorem \ref{theo:stability}). Eq.~\eqref{eq:atomic-origin} is natural to consider from the point of view of the expression for the Kronecker coefficient given in Eq.~\eqref{eq:vpf-to-kron}. In this case, the contribution of the alternant term associated to the identity permutation is $1$, and the contribution of all other alternant terms is $0$ (and so the atomic Kronecker coefficient and Kronecker coefficient are both equal to $1$). These results are described in Section \ref{sec:stability}.

 The atomic Kronecker coefficient can be bounded from above using binomial coefficients (Theorem \ref{theo:atomic-bounds}). By bounding each of the terms of Eq.~\eqref{eq:vpf-to-kron} we are able to obtain upper bounds for the Kronecker coefficients which seem to be best known in certain cases. This is described in Section \ref{sec:kc-bounds}, and the main results are Corollaries \ref{cor:kc-bounds-general} and \ref{cor:kc-bounds}.
 
 Finally, in Section \ref{sec:conclusion} we summarize some open questions.

\section{Vector partition functions and Kronecker coefficients}
\label{sec:VPF}
The central formula, Eq.~\eqref{eq:vpf-to-kron}, was developed by Mishna, Rosas and Sundaram~\cite{MiRoSu21}. It is deduced from the formula using Schur polynomials, determinant formulas for Schur polynomials and, a variable substitution.  We reproduce some of the details here to establish notation. 

\subsection{From Schur polynomials to vector partition generating functions}
We recall the \emph{staircase partition} $\delta^{(k)} = (k-1, k-2, \dots, 1, 0)$. For $\lambda$ with $\ell(\lambda) \leq k$, the \emph{alternant} $a_{\lambda}(x_1, \dots, x_k)$ is the anti-symmetric polynomial
\begin{equation}
    a_{\lambda}(x_1, \dots, x_k) := \det(x_i^{\lambda_j})_{1 \leq i,j \leq k}.
\end{equation}

An expression for the Kronecker coefficients involving alternants is
\begin{equation}\label{eq:KCusingAlt}
    \frac{a_{\delta_n[X]}a_{\delta_n[Y]}}{a_{\delta_{mn}[XY]}}a_{\lambda + \delta_{mn}}[XY] = \sum_{\mu, \nu} g_{\lambda, \mu, \nu}S(a_{\mu + \delta_{m}}[X])S(a_{\nu + \delta_{n}}[Y]),
\end{equation}
where $X = (1, x_1, \dots, x_{m-1}), Y =  (1, y_1, \dots, y_{n-1}), XY = (1, x_1, \dots, x_{m-1}, y_1, \dots, y_{n-1}, x_1y_1, x_1y_2, \dots, x_{m-1}y_{n-1})$, and
\begin{equation}
    S(a_{\alpha}(z_1, \dots, z_k)) = \prod_{i=1}^{k}{z_i^{\alpha_i}},
\end{equation}
for a partition $\alpha$ of length at most $k$. 

The ratio of alternants $\frac{a_{\delta_m}[X]a_{\delta_m}[Y]}{a_{\delta_{mn}}[XY]}$
simplifies to the rational function 
\begin{equation}
    \frac{a_{\delta_m}[X]a_{\delta_n}[Y]}{a_{\delta_{mn}}[XY]} = \frac{1}{\mathcal{A}\mathcal{B}\mathcal{C}\mathcal{D}\mathcal{E}\mathcal{F}}
\end{equation}
with the following polynomials:
\begin{align}
    \mathcal{A} &= \prod_{j=1}^{n} \prod_{i=1}^{m} (x_i - y_j)\\ \label{eq:A}
    \mathcal{B} &= \prod_{j=1}^{n}\prod_{i=1}^{m} (1-x_iy_j) \\
    \mathcal{C} &= \prod_{i=1}^{m-1} x_i^{n-1} \prod_{j=1}^{n-1} y_j^{m-1} \prod_{i=1}^{m-1} (1 - x_i) \prod_{j=1}^{n-1} (1-y_j)^{m-1} \\
    \mathcal{D} &= \prod_{k=1, k \neq i}^{m-1} \prod_{i=1}^{m-1}\prod_{j=1}^{n-1} (x_k - x_iy_j) \prod_{k=1, k \neq i}^{m-1} \prod_{i=1}^{m-1}\prod_{j=1}^{n-1} (y_k - x_iy_j) \\
    \mathcal{E} &= \prod_{j \neq l = 1}^{n-1}\prod_{1 \leq i < k \leq m-1} (x_iy_j - x_ky_l) \\
    \mathcal{F} &= \prod_{i=1}^{m-1} x_i^{\binom{n-1}{2}} \prod_{j=1}^{n-1} y_j^{\binom{m-1}{2}} \prod_{1 \leq i < k \leq m-1}(x_i - x_k)^{n-1} \prod_{1 \leq j < l \leq n-1} (y_j-y_l)^{m-1}. \label{eq:F}
\end{align}
After the variable substitution 
\begin{align}
  x_i &= s_1s_2\dots s_i(t_1t_2\dots t_{n-2})^i \quad \text{ for } 1 \leq i \leq m-1, \label{eq:s-change} \\
  \text{and } y_j &= s_0s_1\dots s_{m-1} (t_1t_2\dots t_{n-2})^{m-1} t_1t_2\dots t_{j-1} \quad \text{ for } 1 \leq j \leq n-1 \label{eq:t-change}
\end{align}
the rational function $\frac{1}{\mathcal{A}\mathcal{B}\mathcal{C}\mathcal{D}\mathcal{E}\mathcal{F}}$ can be written as the product 
\[
\mathbf{s}^{\alpha}\mathbf{t}^{\beta} F_{m,n}(s_0, s_1, \dots, s_{m-1}, t_1, \dots, t_{n-2}) 
\]
where $F_{m,n}$ is a vector partition generating function in the variables $s_0, s_1, \dots, s_{m-1}, t_1, \dots, t_{n-2}$. After the variable substitution, the terms $S(a_{\mu + \delta_m}[X])$ and $S(a_{\nu + \delta_n}[Y])$ become $\mathbf{s}^{r_s(\mu, \nu)}$ and $\mathbf{t}^{r_t(\mu, \nu)}$ respectively. Finally, the term of the determinant \[a_{\lambda + \delta_{mn}}(1, s_0, \dots, s_{m-1}, t_1, \dots, t_{n-2})\] corresponding to permutation $\sigma$ becomes~$\mathbf{s}^{l_s(\lambda; \sigma)}\mathbf{t}^{l_t(\lambda; \sigma)}$.

For a monomial $M$ and variable $x$, by $\deg_x(M)$ we denote the exponent of $x$ in the monomial $M$.
\begin{prop} \label{rem:var-change}
Let $u \in \{1, s_0, \dots, s_{m-1}, t_1, \dots, t_{n-2}\}$. Then
\begin{multline}
    \deg_u(1) \leq \deg_u(x_1) \leq \dots \leq \deg_u(x_{m-1}) \\ \leq \deg_u(y_1) \leq \dots \leq \deg_u(y_{n-1}) \\ \leq \deg_u(x_1y_1) \leq \deg_u(x_1y_2) \leq \dots \leq \deg_u(x_{m-1}y_{n-1}).
\end{multline}
\end{prop}

\subsection{The vector partition functions $p_{A^{m,n}}$}
By $\mathcal{P}_A(\BF{b})$ we denote the set $\mathcal{P}_A(\BF{b}) := \{\BF{x} \in \mathbb{Z}_{\geq 0}^n : A\BF{x} = \BF{b}\}$ of vector partitions of $\BF{b}$, so that $p_A(\BF{b})$ is the cardinality of $\mathcal{P}_A(\BF{b})$. By exploiting some of the properties of the matrices $A^{m,n}$ given in Corollary 30 of~\cite{MiRoSu21} (Properties 1--5 in the list below), we can deduce properties of the corresponding vector partition functions  $p_{A^{m,n}}$ without explicitly computing the associated piecewise quasi-polynomials:
\begin{enumerate}[label=(\roman*)]
    \item each entry of $A^{m,n}$ is a non-negative integer;
    \item the largest entry of $A^{m,n}$ is $2m -1$;
    \item the number of columns of $A^{m,n}$ is ${\binom{mn}{2}} - {\binom{n}{2}} - {\binom{m}{2}}$;
    \item the number of rows of $A^{m,n}$ is $m + n - 2$;
    \item each of the standard basis vectors appears as a column of $A^{m,n}$, and so its rank is $m + n - 2$.
\end{enumerate}

\section{Explicit computation of Kronecker coefficients}
\label{sec:explicit-computation}

When the partition lengths are sufficiently small, it is computationally feasible to determine the vector partition functions needed to compute individual Kronecker coefficients $g_{\lambda, \mu, \nu}$.  We provide explicit formulas for two cases here, starting from Eq.~\eqref{eq:vpf-to-kron}, rewritten below:
\[ g_{\lambda, \mu, \nu} = \sum_{\sigma \in \mathfrak{S}_{mn}} \operatorname{sgn}(\sigma)\ p_{A^{m,n}}\biggr(\mathbf{b}^{m,n}(\lambda, \mu, \nu; \sigma)\biggr).\]
We compute $p_{A^{m,n}}$ first for $m=2,n=3$, then $m=2,n=4$ (the $m=n=2$ case appears in \cite{MiRoSu21}). Remark that, to compute a coefficient, it is best to minimize the choice of $m$ and $n$ that bound the lengths of $\mu$ and $\nu$. The first optimization comes from trying to identify which terms in the sum are zero. Recall, in the $m=n=2$ case, only $7$ of the terms are needed since of the original $4!=24$ terms in the right hand side:  $13$ of them always evaluate to zero for partitions $\lambda, \mu, \nu$, and another $4$ of them cancel pairwise. To eliminate terms in other cases, we consider restrictions imposed by positivity in the linear algebra system, and the partition inequalities on the parts of the partitions.

\subsection{Exact expressions for $g_{\lambda, \mu, \nu}$ when $\ell(\lambda) \leq 6,\ \ell(\mu) \leq 2,\ \ell(\nu) \leq 3$} \label{sec:2-3-6-case}

The matrix $A^{2,3}$ is determined in \cite[Example 5]{MiRoSu21}:
\begin{equation}
A^{2,3} = 
\begin{bmatrix}
    1 & 0 & 0 & 1 & 0 & 0 & 0 & 1 & 1 & 1 & 1 \\
    0 & 1 & 0 & 0 & 1 & 1 & 1 & 1 & 1 & 2 & 2 \\
    0 & 0 & 1 & 1 & 1 & 1 & 2 & 1 & 2 & 2 & 3 \\
\end{bmatrix}
\end{equation}
Using \em Barvinok \em it is straightforward to determine that corresponding vector partition function $p_{A^{2,3}}$ is of degree 8 and has 34 chambers. At most $482$ of the $720$ terms of the alternant $a_{\lambda + \delta_6}$ yield a non-zero contribution to the Kronecker coefficient computation. The most non-zero terms we have found for any given coefficient is $288$. This occurs for $\mu = (99, 99),\ \nu = (66, 66, 66),\ \lambda = (87, 87, 24, 0, 0, 0).$ It is less clear how to find cancelling pairs as in the $m=n=2$ case, so this remains a place for potential optimization - each term represents a vector partition function evaluation, which in the worst case means searching through all chambers. The formula is as follows. 

\begin{prop} 
Let $\lambda, \mu, \nu$ be partitions with $\ell(\lambda) \leq 6,\ \ell(\mu) \leq 2,\ \ell(\nu) \leq 3$. Then the Kronecker coefficient is given by
\begin{multline}
    g_{\lambda, \mu, \nu} = \sum_{\sigma \in \mathfrak{S}_6} \operatorname{sgn}(\sigma) \, \\p_{A^{2,3}}\left(\nu_2 + \nu_3 + 6 - l_s(\lambda; \sigma)_1,\quad \mu_2 + \nu_2 + \nu_3 + 11 - l_t(\lambda; \sigma)_1,\quad \mu_2 + \nu_2 + 2\nu_3 + 13 - l_t(\lambda; \sigma)_2 \right).
\end{multline}
and the atomic Kronecker coefficient is given by
\begin{dmath}
\tilde{g}^{2,3}_{\lambda, \mu, \nu} = p_{A^{2,3}}\left(\nu_2 + \nu_3 - \lambda_3 - \lambda_4 - \lambda_5 - \lambda_6,\quad \mu_2 + \nu_2 + \nu_3 - \lambda_2 - \lambda_3 - \lambda_4 - 2\lambda_5 - 2\lambda_6, \\ \quad \mu_2 + \nu_2 + 2\nu_3 - \lambda_2 - \lambda_3 - 2\lambda_4 - 2\lambda_5 - 3\lambda_6 \right).
\end{dmath}
\end{prop}

The implementation of Baldoni, Vergne and Walter \cite{BaVeWa17} takes on the order of 1 minute to compute a $2,3,6$ Kronecker coefficient whereas our implementation takes on the order of 10 microseconds. However, they are able to compute dilated Kronecker coefficients and, more generally, expressions that hold over the entire chamber, while our code does not do either.

\subsection{Exact expressions for $g_{\lambda, \mu, \nu}$ when $\ell(\lambda) \leq 8,\ \ell(\mu) \leq 2,\ \ell(\nu) \leq 4$} \label{sec:2-4-8-case}

It is straightforward to determine~$A^{2,4}$ following the same method
\begin{equation}
A^{2,4} = \begin{bmatrix}
0 & 0 & 0 & 1 & 0 & 0 & 0 & 1 & 1 & 0 & 0 & 0 & 0 & 0 & 1 & 0 & 1 & 1 & 1 & 1 & 1 \\
0 & 0 & 1 & 0 & 0 & 1 & 1 & 0 & 0 & 1 & 1 & 1 & 1 & 1 & 1 & 1 & 1 & 1 & 2 & 2 & 2 \\
0 & 1 & 0 & 0 & 1 & 0 & 1 & 1 & 1 & 1 & 1 & 1 & 1 & 2 & 1 & 2 & 2 & 2 & 2 & 3 & 3 \\
1 & 0 & 0 & 0 & 1 & 1 & 0 & 0 & 1 & 1 & 1 & 1 & 2 & 1 & 1 & 2 & 1 & 2 & 2 & 2 & 3
\end{bmatrix}.
\end{equation}
The corresponding vector partition function $p_{A^{2,4}}$ is of degree $17$ with $4328$ chambers. It took roughly 20 days to compute it on the Compute Canada  \em Cedar\em\ research cluster. The vector partition function is available in .sobj format and in .txt format. The .txt format is the raw output from \em Barvinok\em.

Out of the $8! = 40320$ terms of the alternant $a_{\lambda + \delta_8}$, at most $28322$ yield a non-zero contribution to the Kronecker coefficient. It is not apparent if they can be grouped for cancellation as in the $m=n=2$ case.

\begin{prop}
Let $\lambda, \mu, \nu$ be partitions with $\ell(\lambda) \leq 8,\ \ell(\mu) \leq 2,\ \ell(\nu) \leq 4$. Then the Kronecker coefficient is given by
 \begin{dmath}
    g_{\lambda, \mu, \nu} = \sum_{\sigma \in \mathfrak{S}_8} \operatorname{sgn}(\sigma)
    p_{A^{2,4}}(\nu_2 + \nu_3 +\nu_4 + 15 -  l_s(\lambda; \sigma)_1, \mu_2 + \nu_2 + \nu_3 + \nu_4 + 24 - l_s(\lambda; \sigma)_2, \\ \mu_2 + \nu_2 + 2\nu_3 + 2\nu_4 + 32 - l_t(\lambda; \sigma)_1, \mu_2 + \nu_2 + \nu_3 + 2\nu_4 + 27 - l_t(\lambda; \sigma)_2),
\end{dmath}
and the atomic Kronecker coefficient is given by
\begin{dmath}
\tilde{g}^{2,4}_{\lambda, \mu, \nu} = p_{A^{2,4}}(\nu_2 + \nu_3 +\nu_4 - \lambda_3 - \lambda_4 - \lambda_5 - \lambda_6 - \lambda_7 - \lambda_8, \\ \mu_2 + \nu_2 + \nu_3 + \nu_4 - \lambda_2 - \lambda_3 - \lambda_4 - \lambda_5 - 2\lambda_6 - 2\lambda_7 - 2\lambda_8, \\ \mu_2 + \nu_2 + 2\nu_3 + 2\nu_4 - \lambda - \lambda_3 - 2\lambda_4 - 2\lambda_5 - 2\lambda_6 - 3\lambda_7 - 3\lambda_8, \\ \mu_2 + \nu_2 + \nu_3 + 2\nu_4 - \lambda_2 - \lambda_3 - \lambda_4 - 2\lambda_5 - 2\lambda_6 - 2\lambda_7 - 3\lambda_8)
\end{dmath}
\end{prop}

\begin{eg}
This formula gives the same result for the following example, take from~\cite{BaOr05}. 
For $\lambda = (6, 4, 4, 1)$, $\mu = (12, 3)$,  $\nu = (5, 4, 3, 3)$, we compute that $g_{\lambda, \mu, \nu} = 4$. The authors of \cite{BaOr05} computed this via a combinatorial rule - in this case the Kronecker coefficient is counting combinatorial objects called \em Kronecker tableaux\em. 
\end{eg}

\begin{eg}
Let $\lambda = (57, 57, 57, 33, 33, 33, 10, 0),\ \mu = (140, 140),\ \nu = (70, 70, 70, 70)$, we compute that $g_{\lambda, \mu, \nu} = 391$. We were unable to compute this example with the package \em SF\em\ (it ran into a memory error after using $203146718216$ bytes), nor the \em Sagemath\em\ symmetric functions package (which also ran into a memory error). It cannot be computed by the Maple package of Baldoni, Vergne, and Walter \cite{BaVeWa17} which specifically handles the cases $\ell(\lambda), \ell(\mu), \ell(\nu) \leq 3$ and $\ell(\lambda) \leq 6, \ell(\mu) \leq 2, \ell(\nu) \leq 3$.
\end{eg}

The $A^{3,3}$ matrix is straightforward to compute, it has 4 rows and 30 columns. However obtaining the piecewise quasi-polynomial representation of the vector partition function was not computationally feasible: we had no results after roughly 30 days on the Compute Canada research cluster Cedar at which time the computation was terminated by the server.

\section{Vanishing conditions} 
\label{sec:gen-bravyi-vanish}
A key to our analysis is a dominance property of  vector partition functions. We use this property to prove Theorem~\ref{theo:gen-bravyi},  a generalization of some non-vanishing conditions for the Kronecker coefficients given in \cite{Bra04}. Let $\BF{u}, \BF{v} \in \mathbb{R}^k$. We say that $\BF{u}$ \emph{dominates} $\BF{v}$ if $\BF{u}_i \geq \BF{v}_i$ for each $1 \leq i \leq k$, and we denote this by $\BF{u} \succeq \BF{v}$. 
\begin{lem} \label{lem:vpf-dominance}
Let $m,n$ be positive integers, if $\BF{a} \succeq \BF{b}$, then $p_{A^{m,n}}(\BF{a}) \geq p_{A^{m,n}}(\BF{b})$ 
\end{lem}

\begin{proof}
Each of the standard basis vectors $\BF{e}_1, \dots, \BF{e}_{m+n-2}$ is a column of $A^{m,n}$. Without loss of generality assume that columns $1, \dots, m+n-2$ are the standard basis vectors $\BF{e}_1, \dots, \BF{e}_{m+n-2}$ in the same order. It follows that any vector partition $\BF{x} \in S_{A^{m,n}}(\BF{b})$ can be mapped to a unique vector partition $\BF{x'} \in S_{A^{m,n}}(\BF{a})$ by taking $\BF{x'}_i := \BF{x}_i + (a_i-b_i)\BF{e}_i$ for each $1 \leq i \leq m + n - 2$. This forms an injective map from~$\mathcal{P}_A(\BF{b})$ to~$\mathcal{P}_A(\BF{a})$.
\end{proof}

\begin{lem} \label{lem:poset-structure}
Let $m,n$ be positive integers. Let $\sigma_1, \sigma_2 \in \mathfrak{S}_{mn}$ such that
\[
(l_s(\lambda; \sigma_1), l_t(\lambda; \sigma_1)) \succeq (l_s(\lambda; \sigma_2), l_t(\lambda; \sigma_2)) 
\]
for all partitions $\lambda$ with $\ell(\lambda) \leq mn$. Then 
\[
p_{A^{m,n}}\biggr(\BF{b}^{m,n}(\lambda, \mu, \nu; \sigma_1)\biggr) \leq p_{A^{m,n}}\biggr(\BF{b}^{m,n}(\lambda, \mu, \nu; \sigma_2)\biggr)
\]
for all partitions $\lambda, \mu, \nu$ with $\ell(\lambda) \leq mn, \ell(\mu) \leq m, \ell(\nu) \leq n$.
\end{lem}

\begin{proof}
    Multiplication by $-1$ reverses domination. The domination of one vector over another is preserved if we subtract the same vector from both sides, and if we add a positive vector to the larger one. Thus, for any partitions $\lambda, \mu, \nu$ with $\ell(\lambda) \leq mn, \ell(\mu) \leq m, \ell(\nu) \leq n$, we find that
\begin{align*}
    \BF{b}^{m,n}(\lambda, \mu, \nu; \sigma_2) &= \biggr(r_s(\mu, \nu), r_t(\mu, \nu)) + (\alpha, \beta) - (l_s(\lambda; \sigma_2), l_t(\lambda; \sigma_2)\biggr) \\
    &\succeq\,%
\biggr(r_s(\mu, \nu),  r_t(\mu, \nu)) + (\alpha, \beta) -(l_s(\lambda; \sigma_1), l_t(\lambda; \sigma_1)\biggr) \\ &= \BF{b}^{m,n}(\lambda, \mu, \nu; \sigma_1).
\end{align*}
Then by Lemma \ref{lem:vpf-dominance} we have that $p_{A^{m,n}}\biggr(\BF{b}^{m,n}(\lambda, \mu, \nu; \sigma_2) \biggr) \geq p_{A^{m,n}}\biggr(\BF{b}^{m,n}(\lambda, \mu, \nu; \sigma_1) \biggr)$ as required. 
\end{proof}

The previous lemma induces a poset structure on $\mathfrak{S}_{mn}$ via the relation $\sigma_2 \geq \sigma_1$ if and only if
\[
(l_s(\lambda; \sigma_1), l_t(\lambda; \sigma_1)) \succeq (l_s(\lambda; \sigma_2), l_t(\lambda; \sigma_2)) 
\]
for all partitions $\lambda$ with $\ell(\lambda) \leq mn$. Figure \ref{fig:poset-22} illustrates the poset for the $m=n=2$ case, showing only the permutations associated to the 7 alternant terms which contribute to the Kronecker coefficient. 
The poset in the figure is the \em dependency digraph for the monomials in $P_{\lambda}$\em\ given in \cite[Figure 4]{MiRoSu21}. However, there the poset is computed by comparing the contributions of the alternant terms as opposed to the vector partition function inputs. Our approach allows us to compute the posets for larger $m,n$ when comparing the contributions is infeasible (either due to the large number of chambers or the difficulty of computing the vector partition function as a piecewise quasi-polynomial). 

In the following lemma we show that the identity permutation is a maximal element of the poset for any positive integers $m,n$ (in fact it is unique, and thus the maximal element).
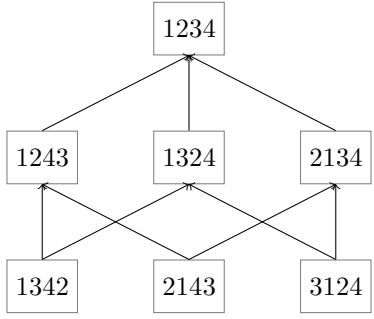
\begin{figure}
    \centering
    
    \begin{tikzpicture}[
myNODE/.style={rectangle,  minimum size=7mm, draw=gray},
]
\node[myNODE]      (ID)  {1234};
\node[myNODE]      (1324) [below=of ID] {1324};
\node[myNODE]      (2134) [right=of 1324] {2134};
\node[myNODE]      (1243) [left=of 1324] {1243};
\node[myNODE]      (3124) [below=of 2134] {3124};
\node[myNODE]      (2143) [below=of 1324] {2143};
\node[myNODE]      (1342) [below=of 1243] {1342};
\draw[->] (3124.north) -- (2134.south);
\draw[->] (3124.north) -- (1324.south);
\draw[->] (2143.north) -- (2134.south);
\draw[->] (2143.north) -- (1243.south);
\draw[->] (1342.north) -- (1243.south);
\draw[->] (1342.north) -- (1324.south);
\draw[->] (2134.north) -- (ID.south);
\draw[->] (1324.north) -- (ID.south);
\draw[->] (1243.north) -- (ID.south);
\end{tikzpicture}

    \caption{The poset of contributing alternant terms in the $m=n=2$ case. Each alternant term is given by its permutation in one line notation.}
    \label{fig:poset-22}
\end{figure}

\begin{lem} \label{lem:domination}
    For all $\sigma \in \mathfrak{S}_{mn}$ and partitions $\lambda$ of length at most $mn$,
    \[(l_s(\lambda; \sigma), l_t(\lambda; \sigma)) \succeq (l_s(\lambda; \id), l_t(\lambda; \id)).\]
\end{lem}

\begin{proof}
    The alternant $a_{\lambda + \delta_{mn}}$ is the determinant of the matrix $(z_i^{\lambda_j})_{1 \leq i,j \leq mn}$ where $z_i$ is the $i$\textsuperscript{th} variable in $XY$. The $k$\textsuperscript{th} coordinate of $(l_s(\lambda; \sigma), l_t(\lambda; \sigma))$ is
    \begin{equation} \label{eq:k-coordinate}
        (l_s(\lambda; \sigma), l_t(\lambda; \sigma))_k = \sum_{i = 1}^{mn} \left(\lambda_i + mn - i \right)\deg_u(z_{\sigma^{-1}(i)})
    \end{equation}
    where $u$ is the $k$\textsuperscript{th} element of $(1, s_0, \dots, s_{m-1}, t_1, \dots, t_{n-2})$. Since $(\lambda_1 + mn - 1, \lambda_2 + mn - 2, \dots, \lambda_{mn})$ is a monotonically decreasing sequence, and $\deg_u$ monotonically increasing over $(1, x_1, \dots, x_{m-1}, y_1, \dots, y_{n-1}, \\ x_1y_1, \dots, x_my_n)$, the above expression is minimized for the term obtained by the change of variables from
    \[
    1^{\lambda_1 + mn - 1}x_1^{\lambda_2 + mn - 2}\dots\ (x_my_n)^{\lambda_{mn}}
    \]
    corresponding to the identity permutation. 
\end{proof}

Combining the previous two lemmas yields the following result relating the atomic Kronecker coefficient $\tilde{g}^{m,n}_{\lambda, \mu, \nu}$ with the Kronecker coefficient $g_{\lambda, \mu, \nu}$, from which vanishing conditions (given in Theorem \ref{theo:gen-bravyi}) can be derived.

\begin{lem} \label{lem:atomic-vanish}
    Let $\lambda, \mu, \nu$ be partitions with $\ell(\lambda) \leq mn, \ell(\mu) \leq m, \ell(\nu) \leq n$ for some positive integers $m,n$.
    If $\tilde{g}^{m,n}_{\lambda, \mu, \nu} = 0$, then $g_{\lambda, \mu, \nu} = 0$.
\end{lem}

\begin{proof}

If $\tilde{g}^{m,n}_{\lambda, \mu, \nu}=0$ then for any $\sigma\in \mathfrak{S}_{mn}, $
\begin{eqnarray*}0=\tilde{g}^{m,n}_{\lambda, \mu, \nu} &=& p_{A^{m,n}}\biggr(\BF{b}^{m,n}(\lambda, \mu, \nu; \id \biggr)\quad\text{by Eq.~\eqref{eq:atomic}}, \\
&\geq& p_{A^{m,n}}\biggr(\BF{b}^{m,n}(\lambda, \mu, \nu; \sigma \biggr) \quad\text{by Lemmas~\ref{lem:poset-structure} and \ref{lem:domination}}, \\
&\geq& 0,  \\
\end{eqnarray*}
and thus $p_{A^{m,n}}\biggr(\BF{b}^{m,n}(\lambda, \mu, \nu; \sigma \biggr) \ = 0$. Since it is true for all $\sigma$, all the terms in the sum in Eq.~\eqref{eq:vpf-to-kron} vanish, and so $g_{\lambda, \mu, \nu} = 0$.
\end{proof}

Since $\tilde{g}^{m,n}_{\lambda, \mu, \nu}$ is given by a single vector partition function evaluation $p_{A^{m,n}}(\mathbf{b})$, we know that it is~$0$ exactly when $\mathbf{b}$ is not in the cone generated by the columns of $A^{m,n}$. This occurs if and only if $b_i < 0$ for some $1 \leq i \leq m+n-2$. Since $\BF{b} = (r_s(\mu, \nu) + \alpha - l_s(\lambda, \id), r_t(\mu, \nu) + \beta - l_t(\lambda, \id))$ for the atomic Kronecker coefficient $\tilde{g}^{m,n}_{\lambda, \mu, \nu}$, we get a set of vanishing conditions for the Kronecker coefficient $g_{\lambda, \mu, \nu}$. We express the conditions using the contrapositive (i.e. we give conditions imposed on $\lambda, \mu, \nu$ if the Kronecker coefficient is non-zero) since the set of $\lambda, \mu, \nu$ satisfying them forms a cone. 

\begin{theo} \label{theo:gen-bravyi} 
    Let $m,n$ be positive integers and $\lambda, \mu, \nu$ be partitions with $\ell(\lambda) \leq mn, \ell(\mu) \leq m, \ell(\nu) \leq n$. If $g_{\lambda, \mu, \nu} \neq 0$ then each of the following inequalities hold:
    \begin{equation}
        \sum_{k=1}^{m} \lambda_k \geq \nu_1;
    \end{equation}
   For all $a$ satisfying $1 \leq a \leq m-1$:
    \begin{equation}
        \sum_{k=1}^{a} \lambda_k - \sum_{k=m+n}^{m + (a+1)(n-1)} \lambda_k \geq \nu_1 - \sum_{k=a+1}^{m} \mu_k
    \end{equation}
 For all $b$ satisfying $1 \leq b \leq n-2$:
{\small\begin{equation} 
m\lambda_1 + \sum_{k=2}^{m} (m-k+1)\lambda_k + \sum_{k=m+1}^{m+b} \lambda_k - \sum_{i=1}^{m-1}\sum_{j=1}^{b}(i-1) \lambda_{m+i(n-1)+j} - \sum_{i=1}^{m-1}\sum_{j=b+1}^{n-1}i\lambda_{m+i(n-1) + j} \geq m\nu_1 + \sum_{k=2}^{b+1} \nu_k - \sum_{k=2}^{m}(k-1)\mu_k.
 \end{equation}}
\end{theo}


\begin{rem}
When $m=n=2$, Theorem \ref{theo:gen-bravyi} reduces to vanishing conditions given by Bravyi in \cite{Bra04}. This case was worked out explicitly in \cite[Proposition 5]{MiRoSu21}. 
\end{rem}
An inequality $n \cdot x \leq 0$ is \emph{essential} for a cone $\tau$ if $\{x : n \cdot x = 0\} \cap \tau$ is a facet of $\tau$, and each $p \in \tau$ satisfies the inequality ($n \cdot p \leq 0$ for all $p \in \tau$). 

\begin{rem}
Klyachko~\cite{Kl04} gives the full list of $41$ essential inequalities in the $m=2, n=3$ case. In this case, none of our inequalities appear on Klyachko's list. Thus, while our inequalities are easy to compute and use practically, regrettably none are essential inequalities for the cone $PKron_{2,3,6}$. Thus, one should not expect, for general $m,n$, that the inequalities given by Theorem \ref{theo:gen-bravyi} are essential.
\end{rem}

Ressayre determined two sets of vanishing conditions for the Kronecker coefficients for any lengths $l,m,n$ which are essential,~\cite[Theorems 1 \& 2]{Re19}. 
\begin{theo}[Ressayre~\cite{Re19}]
    Let $e, f$ be two positive integers, and let $\lambda, \mu, \nu$ be partitions of $N$ with
    \begin{equation}\label{eq:Ress}
    l(\mu) \leq e+1, \ l(\nu) \leq f+1,\ l(\lambda) \leq e + f + 1
    \end{equation}
    If $g_{\lambda, \mu, \nu} \neq 0$, then 
    \[
    N + \lambda_1 + \lambda_{e+j} \leq \mu_1 + \nu_1 + \nu_j 
    \]
    for all $2 \leq j \leq f+1$.
\end{theo}
These are quite strong, although there is likely a smaller error in these conditions, given the following example we found.
\begin{eg}
    Upon setting $e=1, f=3, n=4$ and $j=4$ in Eq.~\eqref{eq:Ress}, $\ell(\mu) \leq 2, \ell(\nu) \leq 4, \ell(\lambda) \leq 5$ and the Kronecker coefficient $g_{\lambda, \mu, \nu}$ should be $0$ if, furthermore,
    \begin{equation}\label{eq:ress2}
        |\lambda| + \lambda_1 + \lambda_5 > \mu_1 + \nu_1 + \nu_4.
    \end{equation}
  
    Consider $\lambda = (1, 1, 1, 1, 0), \mu = (2, 2), \nu = (2, 2)$. Inequality~\ref{eq:ress2} is satisfied, but $g_{\lambda, \mu, \nu} = 1$, not $0$.  A second example is given by  $\lambda = (4), \mu = (2, 2), \nu = (2,2)$. 
    
\end{eg}

\section{A stable face of the Kronecker polyhedron}
\label{sec:stability}
By considering a set of conditions implying that the atomic Kronecker coefficient and Kronecker coefficient are both equal to $1$, we are able to obtain a stable face of the Kronecker polyhedron $PKron_{m,n,mn}$ for each $m,n$. Moreover, each partition triple $(\lambda, \mu, \nu)$ satisfying these conditions is a stable triple. We note that elements of this approach appear in \cite{MiRoSu18} for the case $m=n=2$.

\begin{prop} \label{prop:origin}
    If $\BF{b}^{m,n}(\lambda, \mu, \nu; Id) = \BF{0}$, then $
    g_{\lambda, \mu, \nu} = \tilde{g}^{m,n}_{\lambda, \mu, \nu} = 1.$
\end{prop}

\begin{proof}
    When $\BF{b}^{m,n}(\lambda, \mu, \nu; Id) = \BF{0}$, $\BF{b}^{m,n}(\lambda, \mu, \nu; \sigma)$ has at least one negative coordinate for each $\sigma \in \mathfrak{S}_{mn},\ \sigma \neq Id$, and so 
    \begin{align}
        g_{\lambda, \mu, \nu} &= \tilde{g}^{m,n}_{\lambda, \mu, \nu} \\
        &= p_{A^{m,n}}(\BF{b}^{m,n}(\lambda, \mu, \nu; Id)) \\
        & = 1.
    \end{align}
\end{proof}

The condition $\BF{b}^{m,n}(\lambda, \mu, \nu; Id) = \BF{0}$ yields $m+n-2$ equations involving the parts of $\lambda, \mu, \nu$. By also including the equations $|\lambda| = |\mu| = |\nu|$, we obtain relatively simple expressions for each part of $\mu$ and $\nu$ in the parts of $\lambda$.

\begin{prop} \label{prop:rewrite-eqns}
Let $\lambda, \mu, \nu$ be partitions of the same positive integer $N$ with $\ell(\mu) \leq m, \ell(\nu) \leq n, \ell(\lambda) \leq mn$.
 Then $\BF{b}^{m,n}(\lambda, \mu, \nu; Id) = \BF{0}$ if and only if $(\lambda, \mu, \nu)$ satisfy the following equations:
    \begin{align}
            \mu_u &= \lambda_u +  \sum\limits_{i=m+(u-1)(n-1)+1}^{m + u(n-1)} \lambda_i &\text{ for } u=1, \dots, m  \label{eq:stab-mu} \\
            \nu_1 &= \sum\limits_{i=1}^{m} \lambda_i  & \label{eq:stab-nu1} \\
            \nu_v &= \sum\limits_{i=0}^{m-1} \lambda_{m+(n-1)i + v-1} & \text{ for }  v=2,\dots,n \label{eq:stab-nu2}.
        \end{align}
\end{prop}

The proof of this appears in Appendix~\ref{sec:proof-rewrite-system}. The following result follows directly from Propositions \ref{prop:origin} and \ref{prop:rewrite-eqns}.

\begin{cor} \label{cor:atomic=1}
    Let $\lambda, \mu, \nu$ be partitions of $N$ with $\ell(\mu) \leq m, \ell(\nu) \leq n, \ell(\lambda) \leq mn$, such that $\lambda, \mu, \nu$ satisfy Eqs.~\eqref{eq:stab-mu}--\eqref{eq:stab-nu2}. Then $g_{\lambda, \mu, \nu} = \tilde{g}^{m,n}_{\lambda, \mu, \nu} = 1$. 
\end{cor}

We can say more about the partition triples $(\lambda, \mu, \nu)$ satisfying Eqs.~\eqref{eq:stab-mu}--\eqref{eq:stab-nu2}. We follow \cite{Man15-2} for notation. A triple of partitions $(\lambda, \mu, \nu)$ is called \em weakly stable \em if $g_{k\lambda, k\mu, k\nu} = 1$ for each positive integer $k$. Recall that a triple of partitions $(\lambda, \mu, \nu)$ is \em stable\em\ if for any partitions $\alpha, \beta, \gamma$ the sequence $(g_{\alpha + k\lambda, \beta+k\mu, \gamma + k\nu})_{k \geq 0}$ stabilizes. 

For given positive integers $l,m,n$, the \em weight lattice\em\ $W_{l,m,n}$ is the sublattice of $\mathbb{Z}^{l+m+n}$ defined by the equations $|\lambda| = |\mu| = |\nu|$.  In \cite{Man15-2}, Manivel defines a \em stable face\em\ of the cone $PKron_{l,m,n}$ to be a face of $PKron_{l,m,n}$ whose intersection with $W_{l,m,n}$ is a subset of $SKron_{l,m,n}$ - the set of all weakly stable triples $(\lambda, \mu, \nu)$ with $\ell(\lambda) \leq l,\ \ell(\mu) \leq m, \ell(\nu) \leq n$. A stable face is \em maximal\em\ if it is maximal in $SKron_{l,m,n}$.

Note that the set of triples $(\lambda, \mu, \nu)$ satisfying Eqs.~\eqref{eq:stab-mu}--\eqref{eq:stab-nu2} along with the partition inequalities (for any partition $\alpha$ of length $k$, $\alpha_1 \geq \alpha_2 \geq \dots \geq \alpha_k \geq 0$) generate a cone $\tau_{m,n}$. By Corollary \ref{cor:atomic=1}, each $\lambda, \mu, \nu$ in the intersection $\tau_{m,n} \cap W_{mn, m, n}$ is weakly stable. In fact, as the next theorem shows, they are actually stable.

\begin{theo} \label{theo:stability}
    Each triple $\lambda, \mu, \nu$ satisfying Eqs.~\eqref{eq:stab-mu} -- \eqref{eq:stab-nu2} is a stable triple. Moreover, the cone $\tau_{m,n}$ is a stable face of $PKron_{mn,m,n}$.
\end{theo}

The proof of the previous theorem is given in Appendix \ref{sec:proof-stable-triples}. It relies on the connection between \emph{additive tableaux} and stable faces given in \cite[Propositions 7 and 9]{Man15}.

\begin{eg}
    Let $\lambda = (10, 8, 5, 3, 2, 2), \mu = (17, 12), \nu = (18, 7, 5)$. One can check that $\lambda, \mu, \nu$ satisfy Eqs~\eqref{eq:stab-mu}--\eqref{eq:stab-nu2}. Further we have checked that $g_{k\lambda, k\mu, k\nu} = 1$ for all positive integers $k$ computing the quasi-polynomial $g_{k\lambda, k\mu, k\nu}$ via the code of Baldoni, Vergne and Walter. We now give an example to illustrate the stability of $\lambda, \mu, \nu$. For $\alpha = (34, 27, 20, 12, 4, 3), \beta = (70, 30), \nu = (43, 39, 18)$, the sequence $(g_{\alpha + k\lambda, \beta + k\mu, \gamma + k\nu})_{k \geq 1}$ stabilizes at $44729$ at $k = 6$. The sequence from $k=0$ to $6$ is $2566$, $18028$, $36174$, $43896$, $44638$, $44713$, $44729$.
\end{eg}

We note that the stable face $\tau_{m,n}$ is not maximal in general. For example, $\tau_{3,3}$ is contained in the stable faces $F_2^-, F_5^-, F_7^-$ and $F_8$ from \cite[Example 2]{Man15-2}. In particular, $F_5^-$ is the (maximal) stable facet defined by the intersection of $PKron_{3,3,9}$ and the equation
\begin{equation}
    \mu_2 + 2\mu_3 + 2\nu_2 + 3\nu_3 = \lambda_2 + 2\lambda_3 + 2\lambda_4 + 3\lambda_5 + 3\lambda_6 + 4\lambda_7 + 4\lambda_8 + 5\lambda_9
\end{equation}
which is $b^{3,3}(\lambda, \mu, \nu; Id)_4 = 0$. 

We remark also that a couple well-known results are implied by Theorem \ref{theo:stability}. When $\lambda, \mu, \nu$ are each rectangular partitions of lengths $mn,m,n$ respectively (that is $\lambda_1 = \dots = \lambda_{mn}$,\ $\mu_1 = \dots = \mu_m$,\ $\nu_1 = \dots = \nu_n$), the Kronecker coefficient is $1$ (and the triple ($\lambda, \mu, \nu$) is stable). It is straightforward to check that $\lambda, \mu, \nu$ satisfy Eqs.~\eqref{eq:stab-mu}--\eqref{eq:stab-nu2}. The case $\mu = \lambda$ and $\ell(\nu) = 1$ (so $\nu = (|\lambda|)$) also satisfies the same equations (and again the Kronecker coefficient in this case is $1$, and the partition triple ($\lambda, \mu, \nu$) is stable). 

\section{Upper bounds for Kronecker coefficients} \label{sec:kc-bounds}
The atomic Kronecker coefficients are given by a single vector partition function evaluation $p_{A^{m,n}}(\mathbf{b})$. By constructing a companion matrix to $A^{m,n}$, we are able to obtain a simpler vector partition function for which the evaluations can be computed by hand and whose evaluations bound $p_{A^{m,n}}$ from above. By bounding each of the terms of Eq.\eqref{eq:vpf-to-kron}, we are then able to obtain upper bounds for the Kronecker coefficients. 

\subsection{A bound in terms of atomic Kronecker coefficients}
In~\cite{MiRoSu21}, Mishna, Rosas and Sundaram show that in the $m=n=2$ case, the atomic Kronecker coefficient $\tilde{g}^{2,2}_{\lambda, \mu, \nu}$ bounds the corresponding Kronecker coefficient $g_{\lambda, \mu, \nu}$ from above, and in~\cite{MiRoSu18} they conjecture that this is the case in general. Since we do know that the atomic Kronecker coefficient is the largest term in the sum, we can use this to give a general weaker bound. 

\begin{prop} \label{prop:atomic-bounds-kc}
    Let $\lambda, \mu, \nu$ be partitions with $\ell(\mu) \leq m, \ell(\nu) \leq n, \ell(\lambda) \leq ln$. Then 
    \begin{equation}
        g_{\lambda, \mu, \nu} \leq \frac{(mn)!}{2} \tilde{g}^{m,n}_{\lambda, \mu, \nu}.
    \end{equation}
\end{prop}
\begin{proof}
    Splitting the sum in  Eq.~\eqref{eq:vpf-to-kron} in two, one for the permutations with positive sign, and one for the permutations with negative sign, we bound each of the negative sign terms above by $0$ and each of the positive terms by the atomic term (by Lemmas \ref{lem:vpf-dominance} and \ref{lem:domination}). 
\end{proof}

\subsection{Estimating atomic Kronecker Coefficients}
We can approximate the partition function of a matrix $A$ by replacing its columns with standard basis vectors so that the rank is preserved. Partition functions of such matrices are easy to write using binomial coefficients. Lemma~\ref{lem:col-replace} describes the replacement process and Proposition~\ref{prop:binomial-bound} is the resulting bound. The following proposition sets up Lemma~\ref{lem:col-replace}.

\begin{prop} \label{prop:col-replace}
Let $A$ be a $d \times n$ matrix with integer entries and $\ker(A) \cap \mathbb{R}^{m}_{\geq 0} = \{\BF{0}\}$. Let $\BF{c}$ be a column of $A$, and let $\BF{c'}$ be a $1 \times n$ vector. Let $A'$ be the matrix obtained by replacing column $\BF{c}$ with $\BF{c'}$. If $p_{A'}(\BF{c}) \geq p_{A}(\BF{c})$, then
\[
p_{A'}(\BF{b}) \geq p_{A}(\BF{b})
\]
for all $\BF{b} \in \mathbb{Z}_{\geq 0}^{d}$.
\end{prop}

\begin{proof}
Let $j$ be the index at which column $\BF{c}$ appears in A (and thus column $\BF{c'}$ appears in $A'$). 
Partition the set of vector partitions $\mathcal{P}_A(\BF{b})$ of $\BF{b}$ into $U_1 := \{\BF{x} : A\BF{x} = \BF{b}, x_j = 0\}$ and $U_2 := \{\BF{x} : A\BF{x} = \BF{b}, x_j > 0\}$. The set $U_1$ is equal to the set $\{x : A'\BF{x} = \BF{b}, x_j = 0\}$. Also $\BF{c} \in \{A'\BF{x} : x_j > 0\}$ since $p_{A'}(\BF{c}) \geq p_{A}(\BF{c})$, and so $|U_2| \leq |\{x : A'\BF{x} = \BF{b}, x_j > 0\}|$. Thus $p_A(\BF{b}) = |U_1| + |U_2| \leq p_{A'}(\BF{b})$ as required.
\end{proof}

The following Lemma describes how to replace columns of $A^{m,n}$ with standard basis vectors via the previous proposition. 

\begin{lem} \label{lem:col-replace}
Let $A$ be a $d \times n$ matrix with non-negative integer entries and each standard basis vector $\BF{e}_1, \dots, \BF{e}_d$ appearing as a column of $A$. Let $\BF{c}$ be a column of $A$, and let $I = \{k : c_k > 0, 1 \leq k \leq n\}$ be the set of non-zero coordinates of $\BF{c}$. Let $E^{(i)}$ denote the matrix obtained by replacing column $\BF{c}$ with $\BF{e}_i$ for some $i \in I$. Then
\[
p_{E^{(i)}}(\BF{b}) \geq p_A(\BF{b})
\]
for all $\BF{b} \in \mathbb{Z}_{\geq 0}^{d}$.
\end{lem}

\begin{prop} \label{prop:binomial-bound}
Let~$E$ be a~$d \times n$ matrix such that the columns of $E$ are formed by taking $i_j$ copies of each standard basis vector $\BF{e}_i$ where $i_1, \dots, i_d \geq 0$. Then
\[
p_A(\BF{b}) = \prod_{i=1}^{k} {\binom{b_i + i_j - 1}{i_j - 1}}.
\]
\end{prop}
\begin{proof}
For each component $i$ we must take a total of $b_i$ copies of the standard basis vector $\BF{e}_i$. We can think of this problem as distributing $b_i$ balls to the $i_j$ different columns of $A$ which are the copies of $\BF{e}_i$. This is counted by the $i$\textsuperscript{th} term in the given product of binomial coefficients.
\end{proof}

By application of Lemma~\ref{lem:col-replace} and Proposition \ref{prop:binomial-bound} we obtain binomial coefficient bounds for the atomic Kronecker coefficients, and thus the Kronecker coefficients as well. The technical details of the proof appear in Appendix~\ref{sec:appendix2}  where we work out explicitly which columns have which non-zero coordinates. Note that there are many choices of column replacements that can be made, and different choices provide better bounds for certain choices of $\lambda, \mu, \nu$. The formulation of Theorem \ref{theo:atomic-bounds} represents a single choice whose advantage is that it is relatively simple to explain. 

\begin{theo} \label{theo:atomic-bounds}
Let $m,n$ be positive integers, and $\lambda, \mu, \nu$ be partitions with $\ell(\lambda) \leq mn, \ell(\mu) \leq m,\ \ell(\nu) \leq n$. Then:
\begin{equation} \label{eq:atomic-bounds}
\tilde{g}^{m,n}_{\lambda, \mu, \nu} \leq {\binom{b_1 + c_1}{b_1}}{\binom{b_2 + c_2}{b_2}}{\binom{b_{m+n-2} + c_3}{b_{m+n-2}}}\prod_{i=3}^{m} {\binom{b_i + f_1(i)}{ b_i}}\prod_{j=1}^{n-3} {\binom{b_{m+j} + f_2(j)}{b_{m+j}}}.
\end{equation}
where 
$\BF{b} = (b_1, \dots, b_{m+n-2}) = \mathbf{b}^{m,n}(\lambda, \mu, \nu; Id)$ and 
\begin{align}
    c_1 &= (m^2-1)(n-1) - 1\\
    c_2 &= (m-1)(n-1)^2 - 1 \\ 
    c_3 &= {\binom{m-1}{2}}(n-1) + (m-1) - 1 \\
    f_1(i) &= 2{\binom{n-1}{2}}(i-2) - 1 \\
    f_2(j) &= (n- j - 1)(m-1) - 1
\end{align}
\end{theo}

\begin{cor}  \label{cor:kc-bounds}
Theorem \ref{theo:atomic-bounds} in combination with Proposition~\ref{prop:atomic-bounds-kc} gives:
\[ g_{\lambda, \mu, \nu} \leq  \frac{(mn)!}{2}R_1\] where $R_1$ is the expression on the right-handside of Inequality~\eqref{eq:atomic-bounds}.
\end{cor}

We also give a weaker general bound which depends only on $m,n$ and the size $N$ of the partitions $\lambda, \mu, \nu$. We do this by bounding the coordinates of $\mathbf{b}^{m,n}$ by multiples of $N$. 

\begin{cor} \label{cor:atomic-kc-bounds}
    Let $m,n$ be positive integers, and $\lambda, \mu, \nu, $ be partitions of $N$ with lengths at most $mn,m, n$ respectively.
    \begin{align}
        \tilde{g}^{m,n}_{\lambda, \mu, \nu} \leq {\binom{N + c_1}{N}}{\binom{2N + c_2}{2N}}{\binom{(2m-1)N + c_3}{(2m-1)N}} \prod_{i=3}^{m}{\binom{2N + f_1(i)}{2N}} \prod_{j=1}^{n-3}{\binom{(2m-1)N + f_2(j)}{ (2m-1)N}}.
         \label{eq:atomic-easybound} 
    \end{align}
    where $c_1, c_2, c_3, f_1, f_2$ are as in Theorem \ref{theo:atomic-bounds}.
\end{cor}

\begin{proof}
    Recall that for each component of $\mathbf{b}^{m,n}(\lambda, \mu, \nu; Id) = (r_s(\mu, \nu) + \alpha - l_s(\lambda; Id), r_t(\mu, \nu) + \beta - l_t(\lambda; Id)$ the constant terms cancel.  Therefore each $b_i$ of Theorem \ref{theo:atomic-bounds} is bounded above by linear combination in the parts of $\mu, \nu$ appearing. Explicitly, we find that $b_1 \leq N$, $b_i \leq 2N$ for $2 \leq i \leq m$ and $b_{m+j} \leq (2m-1)N$ for $1 \leq j \leq n-2$.
\end{proof}

As before, combining the previous result with Proposition \ref{prop:atomic-bounds-kc}, we obtain the following bound for the Kronecker coefficients. 

\begin{cor}  \label{cor:kc-bounds-general}
Let $m,n$ be positive integers, and $\lambda, \mu, \nu$ be partitions of $N$ of lengths at most $mn,m,n$ respectively. Then
\[ g_{\lambda, \mu, \nu} \leq  \frac{(mn)!}{2}R_2\] \label{eq:kc-easybound} where $R_2$ is the expression on the right-hand side of Inequality~\eqref{eq:atomic-easybound}.
\end{cor}

The bound given in line~\eqref{eq:kc-easybound} is $O(N^d)$, where $d$ is the difference between the number of columns and rows of $A^{m,n}$ - that is:
\begin{equation}
    d = {\binom{mn}{2}} - {\binom{n}{2}} - {\binom{m}{2}} - n - m + 2
\end{equation}
whereas the bound given in \cite{PaPa20} is $O(N^{(mn)^2})$. Note that this analysis holds for the case when $\ell(\mu) = m, \ell(\nu) = n, \ell(\lambda) = mn$. For example, the bound given by Pak and Panova is stronger if $\ell(\lambda) = \ell(\mu) = \ell(\nu) = m$ since in this case
their bound is $O(N^{m^3})$, while our bound is $O(N^{{\binom{m^2}{2}} - 2{\binom{m}{2}} - 2m + 2})$. If $\ell(\mu) = m,\ \ell(\nu) = n$ are fixed, we find that the exponent $x$ given by our $O(N^x)$ expression is smaller when
\begin{equation} \label{eq:exponent-compare}
    \ell(\lambda) > \frac{mn}{2} - \biggr(\frac{m^2 + n^2 + m + n - 4}{2mn}\biggr) - \frac{1}{2}
\end{equation}
and larger when the inequality is flipped. Note that for $\ell(\mu), \ell(\nu) \geq 2$ (i.e. $m,n \geq 2$), the expression on the right-hand side of \eqref{eq:exponent-compare} is smaller than $\frac{mn}{2}$.

We give the explicit bound in the $m=n=3$ case for which there is no efficient computational tool.

\begin{cor} \label{cor:atomic-bounds} For all partitions $\lambda, \mu, \nu$ of $N$ with $\ell(\mu), \ell(\nu) \leq 3, \ell(\lambda) \leq 9$.  
\begin{equation}
g_{\lambda, \mu, \nu} \leq \frac{9!}{2} {\binom{b_1 + 15}{15}}{\binom{b_2 + 7}{7}}{\binom{b_3 + 1}{1}}{\binom{b_4 + 3}{3}}
\end{equation}
where 
\begin{align*}
    b_1 &= \nu_2 + \nu_3 -\lambda_{4} - \lambda_{5} - \lambda_{6} - \lambda_{7} - \lambda_{8} - \lambda_{9} \\
    b_2 &= \mu_{2} + \mu_{3} + \nu_{2} + \nu_{3} -\lambda_{2} - \lambda_{3} - \lambda_{4} - \lambda_{5} - 2\lambda_{6} - 2\lambda_{7} - 2\lambda_{8} - 2\lambda_{9} \\
    b_3 &= \mu_3 + \nu_2 + \nu_3 -\lambda_{3} - \lambda_{4} - \lambda_{5} - \lambda_{6} - \lambda_{7} - 2 \, \lambda_{8} - 2 \, \lambda_{9} \\
    b_4 &= \mu_{2} + 2 \, \mu_{3} + 2 \, \nu_{2} + 3 \, \nu_{3} -\lambda_{2} - 2 \, \lambda_{3} - 2 \, \lambda_{4} - 3 \, \lambda_{5} - 3 \, \lambda_{6} - 4 \, \lambda_{7} - 4 \, \lambda_{8} - 5 \, \lambda_{9}.
\end{align*} 
\end{cor}

\begin{eg}
Table~\ref{tab:bound} presents  bounds on $g_{\lambda, \mu, \nu}$ where
$\lambda = (15, 15, 15, 10, 10, 10, 10, 10, 5)$, $\mu = (35, 35, 30)$, $\nu = (40, 30, 30)$.

\begin{table}[h]
\begin{tabular}{lcc}\toprule
Source & Bound \\ \midrule
 Corollary \ref{cor:kc-bounds} &  $1.42 \cdot 10^{16}$\\
 Corollary \ref{cor:kc-bounds-general} & $5.38 \cdot 10^{45}$ \\
  Pak and Panova, Inequality \ref{eq:pp-bound1} \cite{PaPa14} & $2.84 \cdot 10^{27}$\\
  Pak and Panova, Inequality \ref{eq:pp-bound2} \cite{PaPa20} & $1.13 \cdot 10^{54}$\\ \bottomrule 
\end{tabular} 

\bigskip

  \caption{Upper bound comparison for $g_{15^3\,10^5\,5,35^2\,30,40\,30^2}$}
  \label{tab:bound}
  \end{table}
  
The bound given by Inequality~\ref{eq:pp-bound1} by Pak and Panova is better on some examples. From our experience our bound is the better choice when~$\lambda$ is close to rectangular due to the large coefficients on small parts of~$\lambda$.
\end{eg}

\section{Conclusion and open problems} \label{sec:conclusion}
The partition function approach to Kronecker functions is elementary, yet provides a useful structure to calculate values and upper bounds. The notion of the atomic Kronecker Coefficient is very useful to determine vanishing conditions, bounds, and also to generate stable triples. We summarize a few open problems.

\tocless\subsection*{\bf 1.} Proving that the atomic Kronecker coefficient $\tilde{g}^{m,n}_{\lambda, \mu, \nu}$ is an upper bound for the Kronecker coefficient $g_{\lambda, \mu, \nu}$ for general partitions. Proving this would allow us to remove the factorial growth in the lengths $m,n$ (the $\frac{(mn)!}{2}$ term). 
    
\tocless\subsection*{\bf 2.} Determine which alternant terms make a non-zero contribution to the Kronecker coefficient, and find cancelling terms (i.e. determine the minimal number of $\sigma \in \mathfrak{S}_{mn}$ needed to sum over in Eq.~\eqref{eq:vpf-to-kron}). This would speed up the computation of the Kronecker coefficients since it reduces the number of vector partition function evaluations necessary.
    
\tocless\subsection*{\bf 3.} For $\lambda, \mu, \nu$ with $\ell(\mu), \ell(\nu) \leq 2, \ell(\lambda) \leq 4$ one can compute $g_{\lambda, \mu, \nu}$ via Theorem~\ref{theo:vpf-to-kron} with $m=n=2$. One can also apply this computation with $m=2,\ n=3$ (or any choice of $m,n$), however in this case one gets many more alternant terms. Is there a way to exploit this in order to simplify the expression in the $m=2,\ n=3$ case (and in general)? 
    
\tocless\subsection*{\bf 4.} Explore the structure of the poset in the discussion preceding Lemma \ref{lem:domination} for general $m,n$.

\tocless\subsection*{\bf 5.} Compute $p_{A^{3,3}}$ as a piecewise quasi-polynomial.

\section{Acknowledgments}
We are extremely grateful to have been given access to the computational resources provided by WestGrid (www.westgrid.ca) and Compute Canada Calcul Canada (www.computecanada.ca) We also thank John Hebron for important technical assistance locally.
Both authors benefited from the financial support of the the National Science and Engineering Research Council (NSERC) of Canada, via NSERC Discovery Grant R611453. We also wish to explicitly acknowledge the important support of Mercedes Rosas and Sheila Sundaram throughout this project. Their feedback, commentary and insights have been essential.   

\bibliographystyle{abbrv} 
\bibliography{Main}
\appendix

\section{Proof of Theorem \ref{theo:atomic-bounds}} \label{sec:appendix2}

We begin by restating Theorem \ref{theo:atomic-bounds}.

\begin{theo} 
Let $m,n$ be positive integers, and $\lambda, \mu, \nu$ be partitions with $\ell(\lambda) \leq mn, \ell(\mu) \leq m,\ \ell(\nu) \leq n$. Then:
\begin{equation} \label{eq:atomic-bounds-appendix}
\tilde{g}^{m,n}_{\lambda, \mu, \nu} \leq {\binom{b_1 + c_1}{b_1}}{\binom{b_2 + c_2}{b_2}}{\binom{b_{m+n-2} + c_3}{b_{m+n-2}}}\prod_{i=3}^{m} {\binom{b_i + f_1(i) }{b_i}}\prod_{j=1}^{n-3} {\binom{b_{m+j} + f_2(j)}{b_{m+j}}},
\end{equation}
where 
$\BF{b} = (b_1, \dots, b_{m+n-2}) = \mathbf{b}^{m,n}(\lambda, \mu, \nu; Id)$ and 
\begin{align*}
    c_1 &= (m^2-1)(n-1) - 1\\
    c_2 &= (m-1)(n-1)^2 - 1 \\ 
    c_3 &= {\binom{m-1}{2}}(n-1) + (m-1) - 1 \\
    f_1(i) &= 2{\binom{n-1}{2}}(i-2) - 1 \\
    f_2(j) &= (n- j - 1)(m-1) - 1
\end{align*}
\end{theo}
\begin{proof}
For each of column $\mathbf{c}$ of $A^{m,n}$ we analyze which of the indices $1 \leq h \leq m+n-2$ are non-zero, in order to understand which of the standard basis vectors $e_1, \dots, e_{m+n-2}$ we may use to replace $\mathbf{c}$ with. The columns of $A^{m,n}$ arise from the binomials of $\mathcal{A}, \dots, \mathcal{F}$ from lines \eqref{eq:A} -- \eqref{eq:F} after the substitution to $s,t$ variables. Explicitly, if column $\mathbf{c}$ of $A^{m,n}$ corresponds to the binomial $1 - \prod\limits_{u=0}^{m-1}\prod\limits_{v=1}^{n-2} s_u^{p_u}t_v^{r_v}$ for non-negative integers $p_0, \dots, p_{m-1}, r_1, \dots, r_{n-2}$, then 
\[
c_k =
\begin{cases}
p_{k-1} \text{ if } 1 \leq k \leq m \\
r_{k-m} \text{ if } m+1 \leq k \leq m+n-2.
\end{cases}.
\]
If $c_k$ is non-zero, then by Proposition \ref{lem:col-replace}, we can replace column $\mathbf{c}$ by $e_k$. 
In the following discussion, we analyze each column of $A^{m,n}$ to find which standard basis vectors may be used to replace it in order to obtain bounds. The information is collected in Table \ref{tab:non-zero-indices-Amn}. Note that we only three cases -- when replacement can be done via the standard basis vector $e_1$, when it can be done via $e_2$, or when any of the standard basis vectors $e_{m+b}$ for $b=1,\dots,n-2$ can be used. We use this approach in order to keep the number of cases relatively low.

\begin{table}[h]
    \centering
    \begin{tabular}{ccccccc}
    \toprule
    column origin & binomial & \# columns & $e_1$ & $e_2$ & $e_{m+b}$  \\
     \midrule
    $\mathcal{A}$ &  $1-\frac{y_j}{x_i}$ & $(m-1)(n-1)$ & \checkmark & & \\
    
    $\mathcal{B}$ & $1 - x_iy_j$ & $(m-1)(n-1)$ & \checkmark & \checkmark & \checkmark    \\
    
    $\mathcal{C}$ & $1 - x_i$ & $(m-1)(n-1)$ & & \checkmark & \checkmark  \\
    
    $\mathcal{C}$ & $1 - y_j$ & $(m-1)(n-1)$ & \checkmark & \checkmark & \checkmark  \\
    
    $\mathcal{D}$ & $1 - \frac{x_iy_j}{x_k}$ & $2{\binom{m-1}{2}}(n-1)$ & \checkmark & \checkmark & \checkmark  \\
    
    $\mathcal{D}$ & $1 - \frac{x_iy_j}{y_k}$ & $2{\binom{n-1}{2}}(m-1)$ & & \checkmark  &   \\
    
    $\mathcal{E}$ & $1 - \frac{x_ky_l}{x_iy_j}$ & $2{\binom{n-1}{2}}{\binom{m-1}{2}}$ & & & \\
    
    $\mathcal{F}$ & $1 - \frac{x_k}{x_i}$ & $(n-1){\binom{m-1}{2}}$ & &  & \checkmark  \\
    
    $\mathcal{F}$ & $1 - \frac{y_l}{y_j}$ & $(m-1){\binom{n-1}{2}}$ & &  &   \\
     
    \bottomrule
    \end{tabular}
    \bigskip
    \caption{Columns of $A^{m,n}$ and the standard basis vectors which can be used to replace them.}
    \label{tab:non-zero-indices-Amn}
\end{table}

For each column type (defined by the form of the binomial it arose from), we provide the number of such columns and illustrate which of the standard basis vectors in the set $\{e_1, e_2\} \cup \{e_{m+b}\}_{1 \leq b \leq n-2}$ can be used to replace such a column. 

\subsubsection*{Columns arising from $\mathcal{A}$} The binomials of $\mathcal{A}$ are of the form 
$1 - \frac{y_j}{x_i}$ for  $1 \leq i \leq m-1,\ 1 \leq j \leq n-1$.
There are $(m-1)(n-1)$ columns of this type. After the variable substitution, the binomial in $s,t$ arising from a given $i,j$ is 
\[
1 - s_0(s_{i+1}\dots s_{m-1})(t_1 \dots t_{n-2})^{m-1-i}(t_1\dots t_{j-1}).
\]
Here $s_0$ appears for each choice of $i,j$. Therefore we can replace any of the columns arising from $\mathcal{A}$ by $e_1$. We thus choose $e_1$ to replace all such columns.

\subsubsection*{Columns arising from $\mathcal{B}$} The binomials of $\mathcal{B}$ are of the form $1 - x_iy_j$ for $1 \leq i \leq m-1,\ 1 \leq j \leq n-1$.
 There are $(m-1)(n-1)$ columns of this type. After the variable substitution, the binomial in $s,t$ arising from a given $i,j$ is 
\[
1 - (s_1 \dots s_i)(t_1\dots t_{n-2})^i (s_0 \dots s_{m-1})(t_1 \dots t_{n-2})^{m-1} (t_1 \dots t_{j-1}).
\]
Each $s_a$ and $t_b$, ($0 \leq a \leq m-1,\ 1 \leq b \leq n-2$) appear for all choices of $i,j$. Therefore we can replace any of the columns arising from $\mathcal{B}$ by any of the standard basis vectors $e_1, \dots, e_{m+n-2}$. We choose $e_1$ to replace all such columns. 

\subsubsection*{Columns arising from $\mathcal{C}$}
There are two types of binomials arising from $\mathcal{C}$.

The first type of binomial of $\mathcal{C}$ is of the form $1 - x_i$ for  $1 \leq i \leq m-1$.
    In this case each binomial is raised to the power $n-1$, so there are $(m-1)(n-1)$ columns of this type. After the variable substitution, the binomial in $s, t$ arising from a given $i$ is
    \[
    1 - (s_1 \dots s_i)(t_1\dots t_{n-2})^i.
    \]
    Here $s_1$ and $t_b$ ($1 \leq b \leq n-2$) appear for all choices of $i$. Therefore we can replace any of the columns of the first type arising from $\mathcal{C}$ by $e_2$ or $e_{m+b}$ for $1 \leq b \leq n-2$. We choose $e_2$ to replace all such columns. 
    
The second type of binomial of $\mathcal{C}$ is of the form $1 - y_j$ for $1 \leq j \leq n-1$. In this case each binomial is raised to the power $m-1$, so there are $(m-1)(n-1)$ columns of this type. After the variable substitution, the binomial in $s, t$ arising from a given $j$ is 
    \[
    1 - (s_0 \dots s_{m-1})(t_1 \dots t_{n-2})^{m-1}(t_1 \dots t_{j-1})
    \]
    Here each $s_a$, $t_b$ ($0 \leq a \leq m-1,\ 1 \leq b \leq n-2$) appear for all choices of $j$. Therefore we can replace any of the columns of the second type arising from $\mathcal{C}$ by any of the standard basis vectors $e_1, \dots, e_{m+n-2}$. We choose $e_1$ to replace all such columns. 
    
\subsubsection*{Columns arising from $\mathcal{D}$}
There are two types of binomials arising from $\mathcal{D}$. 
The first type of binomial of $\mathcal{D}$ is of the form $1 - \frac{x_iy_j}{x_k}$ for $1 \leq i \leq m-1, 1 \leq k \leq m-1, 1 \leq j \leq n-1$ with $k \neq i$. There are $2{\binom{m-1}{2}}(n-1)$ columns of this type. After the variable substitution, the binomial in $s, t$ arising from a given choice of $i,j,k$ is
    \[
    1 - (s_0 \dots s_i)(s_{k+1} \dots s_{m-1})(t_1 \dots t_{n-2})^{m + i - k - 1}(t_1 \dots t_{j-1})
    \]
    Here $s_0, s_1$ and $t_b$ ($1 \leq b \leq n-2$) each appear for all choices of $i,j,k$. Therefore, we can replace any of the columns of the first type arising from $\mathcal{D}$ by $e_1, e_2$ or $e_{m+b}$ for $1 \leq b \leq n-2$. We choose $e_1$ to replace all such columns. 

The second type of binomial of $\mathcal{D}$ is of the form $1 - \frac{x_iy_j}{y_k}$ for $1 \leq i \leq m-1, 1 \leq j \leq n-1, 1 \leq k \leq n-1$ with $k \neq j$. There are $2{\binom{n-1}{2}}(m-1)$ columns of this type. After the variable substitution, the binomial in $s, t$ arising from a given choice of $i,j, k$ is
    \[
    1 - (s_1 \dots s_i)(t_1 \dots t_{k-1})^{i-1} (t_k \dots t_{n-2})^{i}(t_1 \dots t_{j-1})
    \]
    Here $s_1$ appears for each choice of $i,j,k$. Therefore we can replace any of the columns of the first type arising from $\mathcal{D}$ by $e_2$. We thus choose $e_2$ to replace all such columns. 
\subsubsection*{Columns arising from $\mathcal{E}$}   The binomials of $\mathcal{E}$ are of the form $1 - \frac{x_ky_l}{x_iy_j}$ for $1 \leq i < k \leq m-1, 1 \leq j \leq n-1, 1 \leq l \leq n-1$ with $j \neq l$. There are $2{\binom{n-1}{2}}{\binom{m-1}{2}}$ columns of this type. After the variable substitution, the binomial in $s, t$ arising from a given choice of $i, j, k, l$ is
\[
1 - (s_{i+1} \dots s_k)(t_1 \dots t_{j-1})^{k-i - 1}(t_j \dots t_{n-2})^{k-i}(t_1 \dots t_{l-1})
\]
Here none of the variables appear for each choice of $i,j,k,l$. However, for any $q=2,\dots, m-1$, we may consider the set of rows for which $s_q$ appears and no $s_r$ with $r > q$ appears. In each of these cases using $e_{q+1}$ to replace the row is a sensible choice for any choice of $q$. There are 
\[
2{\binom{n-1}{2}}(q-1)
\]
such columns for each $2 \leq q \leq m-1$ (note that $q \neq 1$ since $i+1 \geq 2$).

\subsubsection*{Columns arising from $\mathcal{F}$}
There are two types of binomials arising from $\mathcal{F}$. 
 The first type of binomial of $\mathcal{F}$ is of the form $1 - \frac{x_k}{x_i}$ for $1 \leq i < k \leq m-1$. Each of these binomials is raised to the power $n-1$, so there are ${\binom{m-1}{2}}(n-1)$ columns of this type. After the variable substitution, the binomial in $s, t$ arising from a given choice of $i,k$ is
    \[
    1 - (s_{i+1} \dots s_k)(t_1 \dots t_{n-2})^{k-i}
    \]
    Here $t_b$ ($1 \leq b \leq n-2$) appears for all choices of $i,k$. Therefore we can replace any of the columns arising from $\mathcal{A}$ by $e_{m+b}$ for $1 \leq b \leq n-2$. We choose $e_{m+n-2}$ to replace all such columns. 

The second type of binomial of $\mathcal{F}$ is of the form $1 - \frac{y_l}{y_j}$ for $1 \leq j < l \leq n-1$. Each of these binomials is raised to the power $m-1$, so there are ${\binom{n-1}{2}}(m-1)$  columns of this type. After the variable substitution, the binomial in $s, t$ arising from a given choice of $j,l$ is 
    \[
    1 - (t_j \dots t_{l-1})
    \]
    Here none of the variables appear for all choices of $j,l$. 
    However, for any $q=1,\dots,n-2$ we may consider the set of rows for which $t_q$ appears but no $t_r$ appears with $r < q$. In each of these cases using $e_{m+q}$ to replace the row is a sensible choice for any choice of $q$. There are 
\[
(n-1-q)(m-1)
\]
such columns for each $1 \leq q \leq n-2$ (note that $q < l$, so $q \neq n-1$).

Replacing each column via the process described above produces the bound given in Theorem~\ref{theo:atomic-bounds}, where $c_1 + 1$ is the number of columns replaced by $e_1$, $c_2 + 1$ is the number of comlumns replaced by $e_2$, $c_3 + 1$ is the number of columns replaced by $e_{m+n-2}$, $f_1(i) + 1$ is the number of columns replaced by $e_{i}$ (for $3 \leq i \leq m$) and $f_2(j) + 1$ is the number of columns replaced by $e_{m+j}$ (for $1 \leq j \leq n-3$). Remark that the added ones appear since $c_1, c_2, c_3, f_1, f_2$ have incorporated the subtraction by one necessary for the negative binomial coefficient. 
\end{proof}

\section{Proof of Proposition~\ref{prop:rewrite-eqns}} \label{sec:proof-rewrite-system}
We begin by restating Proposition~\ref{prop:rewrite-eqns}.

\begin{prop} 
Let $\lambda, \mu, \nu$ be partitions of the same positive integer $N$ with $\ell(\mu) \leq m, \ell(\nu) \leq n, \ell(\lambda) \leq mn$.
 Then $\BF{b}^{m,n}(\lambda, \mu, \nu; Id) = \BF{0}$ if and only if $(\lambda, \mu, \nu)$ satisfy the following equations:
    \begin{align}
            \mu_u &= \lambda_u +  \sum\limits_{i=m+(u-1)(n-1)+1}^{m + u(n-1)} \lambda_i &\text{ for } u=1, \dots, m  \label{eq:stab-mu-appendix} \\
            \nu_1 &= \sum\limits_{i=1}^{m} \lambda_i  & \label{eq:stab-nu1-appendix} \\
            \nu_v &= \sum\limits_{i=0}^{m-1} \lambda_{m+(n-1)i + v-1} & \text{ for }  v=2,\dots,n \label{eq:stab-nu2-appendix}.
        \end{align}
\end{prop}

\begin{proof}
One can check that for $\lambda, \mu, \nu$ respecting Eqs.~\eqref{eq:stab-mu} -- \eqref{eq:stab-nu2} we do indeed get $\BF{b}^{m,n}(\lambda, \mu, \nu; Id) = \BF{0}$. 

The set of solutions $(\lambda, \mu, \nu)$ to $\BF{b}^{m,n}(\lambda, \mu, \nu; Id) = \BF{0}$ over $\mathbb{R}^{m + n + mn}$ with $|\lambda| = |\mu| = |\nu|$ is $\ker(Q)$ for a matrix $Q$ whose rows are given by the equations $|\mu| = |\lambda|$, $|\nu| = |\lambda|$ and the coordinate-wise equalities $\BF{b}^{m,n}(\lambda, \mu, \nu; Id)_i = 0$ for $1 \leq i \leq m + n -2$. Below we give the matrix $Q'$ obtained from $Q$ by removing all columns indexed by $\lambda$. The row corresponding to coordinate $i$ of $\BF{b}^{m,n}(\lambda, \mu, \nu; Id)$ is indexed by the $s$ or $t$ variable from which the equation arises. 

\begin{equation}
Q' = 
    \begin{blockarray}{ccccccccccccc}
    & \mu_1 & \mu_2 & \mu_3 & \dots & \mu_{m-1} & \mu_m & \nu_1 & \nu_2 & \nu_3 & \dots & \nu_{n-1} & \nu_n \\
    \begin{block}{c(cccccccccccc)}
        |\mu| = |\lambda| & 1 & 1 & \dots & \dots & \dots & 1 & 0 & 0 & \dots & \dots & \dots & 0 \\
        |\nu| = |\lambda| & 0 & 0 & \dots & \dots & \dots & 0 & 1 & 1 & \dots & \dots & \dots & 1 \\
        s_0 & 0 & 0 & \dots & \dots & \dots & 0 & 0 & 1 & 1 & \dots & \dots & 1 \\
        s_1 & 0 & 1 & 1 & \dots & \dots & 1 & 0 & 1 & 1 & \dots & \dots & 1 \\
        s_2 & 0 & 0 & 1 & \dots & \dots & 1 & 0 & 1 & 1 & \dots & \dots & 1 \\
        \vdots & \vdots &  &  & & & & \vdots  & & & & & \vdots \\
        s_{m-1} & 0 & 0 & 0 & \dots & 0 & 1 & 0 & 1 & 1 & \dots & \dots & 1 \\
        t_1 & 0 & 1 & 2 & \dots & m-2 & m-1 & 0 & m-1 & m & m & \dots & m \\
        t_2 & 0 & 1 & 2 & \dots & m-2 & m-1 & 0 & m-1 & m-1 & m & \dots & m \\
        \vdots & \vdots &  &  & &  & & \vdots &  & & & & \vdots \\
        t_{n-2} & 0 & 1 & 2 & \dots & m-2 & m-1 & 0 & m-1 & m-1 & m-1 & \dots & m-1 \\
    \end{block}
    \end{blockarray}
\end{equation}

The rank of $Q'$ is $m + n$. Therefore $Q$ also has rank $m+n$, and so $\ker(Q)$ has dimension $mn$ and co-dimension $m+n$. Since the set of $\lambda, \mu, \nu$ respecting Eqs.~\eqref{eq:stab-mu} -- \eqref{eq:stab-nu2} also has co-dimension $m+n$, we see that the two systems of linear equations are equivalent.

\end{proof}

\section{Proof of Theorem \ref{theo:stability}} \label{sec:proof-stable-triples}
    In \cite{Man15}, Manivel gives a description of the stable faces of the Kronecker polyhedron in terms of a particular type of standard tableau. A standard tableau $T$ of shape $m \times n$ is \em additive\em\ if there exist increasing sequences $x_1 < x_2 < \dots < x_m, y_1 < y_2 < \dots < y_n$ with the property that
    \begin{equation}
        T(i,j) < T(l,k) \iff x_i + x_j < x_l + x_k.
    \end{equation}
    For an $m \times n$ additive tableau $T$ and partition $\lambda$ of length at most $mn$, Manivel defines the partitions $a_T(\lambda)$ and $b_T(\lambda)$ as follows:
    \begin{align}
        a_T(\lambda)_i = \sum\limits_{j=1}^{m} \lambda_{T(i,j)} \text{ for } i = 1,\dots,m\\
        b_T(\lambda)_j = \sum\limits_{i=1}^{n} \lambda_{T(i,j)} \text{ for } j = 1,\dots,n.
    \end{align}
    Then $(\lambda, a_T(\lambda), b_T(\lambda))$ is a stable triple \cite[Proposition 7]{Man15} and the set  $\{(\lambda, a_T(\lambda), b_T(\lambda)) : \ell(\lambda) \leq mn \}$ is a face of the Kronecker polyhedron of minimal dimension \cite[Proposition 9]{Man15}. We now restate Theorem \ref{theo:stability}, and then show that $\tau_{m,n}$ can be described by an additive tableau, thus proving that each $\lambda, \mu, \nu$ satisfying Eqs.~\eqref{eq:stab-mu} -- \eqref{eq:stab-nu2} is a stable triple. 
    
    \begin{theo} 
    Each triple $\lambda, \mu, \nu$ satisfying Eqs.~\eqref{eq:stab-mu} -- \eqref{eq:stab-nu2} is a stable triple. Moreover, the cone $\tau_{m,n}$ is a stable face of $PKron_{mn,m,n}$.
\end{theo}
    
    \begin{proof}
    Consider the tableau
    \begin{equation}
        T = \begin{bmatrix}
            1 & m+1 & m+2 & \dots & m+n-1 \\
            2 & m+(n-1) + 1 & m+(n-1)+2 & \dots & m+2(n-1)              \\
            3 & m+2(n-1) + 1 & m+2(n-1)+2 & \dots & m+3(n-1)              \\
            \vdots & \vdots & \vdots & & \vdots      \\
            m & m + (m-1)(n-1) + 1 & m + (m-1)(n-1) + 2 & \dots & mn
        \end{bmatrix}
    \end{equation}
    defined by
    \begin{align}
        T_{i,1} &= i & \text{ for } i=1,\dots,m \\
    T_{i,j} &= m + i(n-1) + j & \text{ for } i=2,\dots,m, j = 1,\dots,n.
    \end{align}
    
    It is straightforward to check that for any $\lambda$ with $\ell(\lambda) \leq mn$, $a_T(\lambda)$ and $b_T(\lambda)$ are the partitions $\mu$ and $\nu$ defined by Eqs.~\eqref{eq:stab-mu} -- \eqref{eq:stab-nu2}. We now show that $T$ is an additive tableau.
    
    Consider the sequences
    \[
    x_i = (i-1)(n-1) \text{ for } i=1,\dots, m
    \] and 
    \[ 
    y_1 = 0, y_j = (m-1)(n-1) + j-1 \text{ for } j=2,\dots,m.
    \]

If $T_{i,j} < T_{k,l}$, we have three main cases to consider. 
\begin{enumerate}
    \item If $l = 1$, then $j=1$, and so $i < k$. In this case $x_i + y_1 < x_k + y_1$ since $x_i < x_k$.
    \item If $l \geq 2$ and $j=1$, then 
    \begin{align*}
    x_i + y_1 &\leq (m-1)(n-1) \\ 
    &< (m-1)(n-1) + 1 \\
    &\leq x_k + y_2 \\\
    &\leq x_k + y_l
    \end{align*}
    \item If $l, j \geq 2$, then $T_{i,j} < T_{k,l}$ if and only if $i < k$ or ($i=k$ and $j < l$). \\
    If $i < k$, then 
    \begin{align*}
    x_i + y_j &= (i-1)(n-1) + (m-1)(n-1) + j - 1 \\
    &< i(n-1) + (m-1)(n-1) + 1 \\
    &\leq x_{i+1} + y_2 \\
    &\leq x_k + y_l
    \end{align*}
so $x_i + y_j < x_k + y_l$. \\
If $i=k$ and $j < l$, then $x_i + y_j < x_k + y_l$ since $y_j < y_l$.
\end{enumerate}
Therefore $T$ is an additive tableau and $\tau_{m,n}$ is the face associated to $T$. Thus, we conclude that each triple of partitions $\lambda, \mu, \nu$ satisfying Eqs.~\eqref{eq:stab-mu} -- \eqref{eq:stab-nu2} is stable, and that $\tau_{m,n}$ is a stable face of $PKron_{mn,m,n}$.
 \end{proof}
 
\begin{rem}
    In \cite{Man15}, Manivel introduces the \em $(T, \lambda)$-reduced Kronecker coefficient \em\ $g_{T,\lambda}(\alpha, \beta, \gamma)$ to be the stable value of the sequence $(g_{\alpha + k\lambda, \beta + k\mu, \gamma + k\nu})_{k \geq 0}$. He also shows that the $(T, \lambda)$-reduced Kronecker coefficient counts integral points in a polytope $P_{T, \lambda}$ (and thus may be written as a vector partition function). It may be interesting to compare $(T, \lambda)$-reduced Kronecker coefficients (for the $T$ given above) and atomic Kronecker coefficients for a given $m,n$ (although the choice of $\lambda$ is not a priori obvious).
\end{rem}

\end{document}